\documentclass[12pt,reqno]{amsart}
\usepackage{amssymb,latexsym,amsmath,epsfig,amsthm,mathrsfs}

\usepackage{comment}
\usepackage{rotating}
\usepackage{graphicx}
\usepackage{amssymb}
\usepackage{lineno}
\usepackage{enumitem}
\usepackage{cite}
\usepackage{multicol}
\usepackage[top=3cm, bottom=3cm, left=2.5cm, right=2.5cm]{geometry}
\usepackage[usenames]{color}
\usepackage[colorlinks=true,
linkcolor=blue,
filecolor=blue,
citecolor=blue]{hyperref}

\numberwithin{equation}{section}

\newtheorem{prop}{Proposition}[section]

\newtheorem{theorem}{Theorem}[section]

\newtheorem{conjecture}[theorem]{Conjecture}

\begin{document}
	
	\title[Extended Inverse results for restricted $h$-fold sumset in  integers]{Extended Inverse results for restricted $h$-fold sumset in  integers}
	
	\author[ Debyani Manna  ]{Debyani Manna}
	\address{Department of Mathematics, Indian Institute of Technology Roorkee, Uttarakhand, 247667, India}
	\email{d\_manna@ma.iitr.ac.in, debyani97@gmail.com}

	\author[Mohan]{Mohan}
	\address{Department of Applied Sciences and Humanities, BK Birla Institute of Engineering and Technology, Pilani, 333031, India}
	\email{mohan98math@gmail.com}
	\author[Ram Krishna Pandey]{Ram Krishna Pandey$^{\dagger}$}
	\address{Department of Mathematics, Indian Institute of Technology Roorkee, Uttarakhand, 247667, India}
	\email{ram.pandey@ma.iitr.ac.in}

	\begin{abstract}
	Let $A$ be a finite set of $k$ integers. For $h \leq k$, the \textit{restricted h-fold sumset} $h^{\wedge}A$ is set of all sums of $h$ distinct elements of the set $A$. In additive combinatorics, much of the focus has traditionally been on finite integer sets whose sumsets are unusually small (cf. Freiman’s theorem and its extensions). More recently, Nathanson posed the inverse problem for the restricted sumset $h^{\wedge}A$ when $|h^{\wedge}A|$ is small. For $h \in \{2,3,4\}$, this question has already been studied by Mohan and Pandey. In this article, we study the inverse problems for $h^{\wedge}A$ with arbitrary $h \geq 3$ and characterize all possible sets $A$ for certain cardinalities of $h^{\wedge}A$.
		
		  
	\end{abstract}
	
	\maketitle
	

	\setcounter{page}{1}
	\section{Introduction}
	Let $\mathbb{Z}$ be the set of all integers, $\mathbb{N}_{0}$ be the set of all nonnegative integers, and $\mathbb{N}$ be the set of all positive integers. For a given set $A$ of integers, $|A|$  denotes the cardinality of the set $A$.  For given   integers $\alpha$ and $\beta$, and a set $A$ of integers, we define
	\begin{align*}
		\alpha + A &:=\{\alpha + a : a \in A\},\\
		\alpha - A &:=\{\alpha - a : a \in A\},\\  
		\alpha \ast A &:=\{\alpha a: a \in A\},\\
		[\alpha,\beta] &:=\{x \in \mathbb{Z} : \alpha \leq x \leq \beta\}.
	\end{align*}
	The greatest common divisor of the given integers $c_{1}, c_{2}, \ldots, c_{k}$ is denoted by $d(c_{1}, c_{2}, \ldots, c_{k})$, in particular, $d(A)$ represents the the greatest common divisor of the elements of the set $A$.
	
	Let $h$ and $k$ be  positive integers. Let $A = \{a_{0}, a_{1}, \ldots,a_{k-1}\}$ be a nonempty finite set of $k$ integers.  The \textit{$h$-fold sumset}, denoted by $hA$,  is defined as follows: $$hA :=\left\lbrace\sum_{i=0}^{k-1} \lambda_i a_i: \lambda_i \in [0,h]~\text{for}~ i\in [0, k-1] ~ \text{with}~ \sum_{i=0}^{k-1} \lambda_i =h \right\rbrace.$$
	The \textit{restricted $h$-fold} sumset $h^{\wedge}A$, denoted by $h^{\wedge}A$, is defined as follows: $$h^{\wedge}A :=\left\lbrace\sum_{i=0}^{k-1} \lambda_i a_i: \lambda_i \in \{0,1\}~\text{for}~ i\in [0, k-1] ~ \text{with}~ \sum_{i=0}^{k-1} \lambda_i =h \right\rbrace.$$ 
	It is easy to observe that to define $h^{\wedge}A$ for any finite set of $k$ integers, it is necessary that the set $A$ contains at least $h$ elements. Therefore, in the case of $h^{\wedge}A$, we assume $h\leq k = |A|$, otherwise $h^{\wedge}A = \emptyset$ if $h>k = |A|$. Additionally, we define $0^{\wedge}A=\{0\}$. For integers $a$ and $b$, it is easy to observe that 
	\[h\left(\left(a \ast A\right) + b \right) =(a \ast \left(hA \right) ) + hb \ 
	\text{and}  \ h^{\wedge}\left(\left(a \ast A\right) + b \right) =(a \ast \left(h^{\wedge}A \right) ) + hb.
	\]
	It follows that  $|hA|$ and $|h^{\wedge}A|$  are dilation and translation   invariant of the set $A$ when $A$ is a finite set of integers. Further, in the case of restricted $h$-fold sumset, we have
	\[(k-h)^{\wedge}A = \left(\sum_{i=0}^{k-1}a_{i}\right)-h^{\wedge}A.\] Therefore, 
	\begin{equation}\label{Eq-1}
		|h^{\wedge}A| = |(k-h)^{\wedge}A|.
	\end{equation}
	
	A fundamental problem associated with sumsets involves investigating the structure and characteristics of the sumset, when the set $A$ is provided, which is called the \textit{direct problem}. An \textit{inverse problem} involves inferring properties of the set $A$ based on the properties of its sumset. To access a comprehensive bibliography on these types of problems, one may refer to \cite{Mann, GAF1973, Nathanson1996, TaoVu2006,Cauchy, Davenport,  ElKerPlagne2003, Plagne06-I, Plagne06-II,Erdos-heibronn64, DSH1994,ANR95,ANR96}. Nathanson proved the following direct and inverse theorems for  $hA$ and $h^{\wedge}A$, when $A$ is a finite set of integers.
	\begin{theorem} \textup{\cite[Theorem 1.3, Theorem 1.6]{Nathanson1996}}  \label{Direct and Inverse Thm Nathanson usual}
		Let $h$ and $k$ be positive integers. Let $A$ be a nonempty finite set of $k$
		integers. Then $$|hA|\geq h|A|-h+1.$$ 	This lower bound is best possible. 
		Furthermore, if $h \geq 2$ and $|hA|=h|A|-h+1,$ then $A$ is a $k$-term arithmetic progression.
	\end{theorem}
	
	\begin{theorem} \textup{\cite[Theorem 1.9, Theorem 1.10]{Nathanson1996}}  \label{Direct Thm Nathanson restricted}
		Let $h$ and $k$ be positive integers with $h \leq k$. Let $A$ be a nonempty finite set of $k$ integers.  Then 
		$$| h^{\wedge}A|  \geq h | A|  - h^{2} + 1.$$
		This lower bound is best possible. Furthermore, if $k\geq 5$, $2\leq h \leq k-2$, and  $|h^{\wedge}A|  = hk - h^{2} + 1,$ then $A$ is a $k$-term arithmetic progression.
	\end{theorem}
	
	An another type of problem associated with sumsets also appeared in the literature called \textit{extended inverse problem}: for a finite set $A$, if the size of its sumset deviates significantly from the minimum size, what conclusions can be drawn about the structure of the set $A$? Freiman was the first who delved in this type of problem and proved the following theorem. 
	\begin{theorem}\textup{\cite[Theorem 1.16]{Nathanson1996}}
		Let $k \geq 3$ be a positive integer.	Let $A$ be a nonempty finite set of $k$ integers. If $$|2A|=2k-1+b \leq 3k-4,$$ then $A$ is a subset of an arithmetic progression of length $k+b \leq 2k-3$.
	\end{theorem} 
	
	\noindent Numerous results concerning the extended inverse theorem on abelian groups were studied by Kemperman\cite{Kemperman60} and Grynkiewicz\cite{Grynkiewicz2009}. Lev\cite{Lev1999} generalized Freiman's theorem for regular $h$-fold sumsets $hA$. Building on Lev's results, Tang and Xing\cite{TangXing2021}, and Mohan and Pandey \cite{MohanPandey2023-I}  studied extended inverse results for $hA$ in a more precise manner.	 Only a few extended inverse results are known for $h^{\wedge}A$. Freiman proposed the following conjecture, and it was independently suggested by Lev\cite{VFLev2000}.
	
	
	\begin{conjecture}\textup{\cite[Conjecture 1]{VFLev2000}}\label{Conjecture 1}
		Let  $A=\lbrace a_{0}, a_{1}, \ldots, a_{k-1}\rbrace$ be a set of $k>7$ integers such that
		$0=a_{0}<a_{1}< \cdots < a_{k-1}$ with $d(A)=1$.  Then
		\[
		\left|2^{\wedge}A\right| \geq 
		\begin{cases}
			a_{k-1} + k-2, &  a_{k-1} \leq 2k-5; \\
			3k-7,   &  a_{k-1} \geq 2k-4.
		\end{cases}
		\]
	\end{conjecture}
	\noindent The lower bounds are optimal, by taking $A = \{0, 1, \ldots , k-3\} \cup  \{a_{k-1} - 1, a_{k-1}\}$, resulting in $2^{\wedge}A = \{1, 2, \ldots , 2k-7\} \cup \{a_{k-1} - 1, \ldots , a_{k-1}+ k-3\} \cup \{2a_{k-1} - 1\}$. Freiman et al. \cite{FreimanLowPit} were the first who did some work in the direction of proving  Conjecture \ref{Conjecture 1}. Later, Lev  \cite{VFLev2000} improved Freiman et al. \cite{FreimanLowPit} result by proving the following theorem.
	\begin{theorem}\textup{\cite[Theorem 1.1]{VFLev2000}}\label{Lev_restricted_thm}
		Let  $A=\lbrace a_{0}, a_{1}, \ldots, a_{k-1}\rbrace$ be a set of $k \geq 3$ integers such that
		$0=a_{0}<a_{1}< \cdots < a_{k-1}$ with $d(A)=1$.  Then
		\[
		\left|2^{\wedge}A\right| \geq  
		\begin{cases}
			a_{k-1} + k-2, &  a_{k-1} \leq 2k-5; \\
			(\theta +1)k-6,   &   a_{k-1} \geq 2k-4.
		\end{cases}
		\]
		where $\theta = \dfrac{(1+\sqrt{5})}{2}$ is the 'golden mean'.
	\end{theorem}
	\noindent Recently, Conjecture \ref{Conjecture 1} almost settled by Daza et al. \cite{Daza2023}. Using the above theorem, Mohan and Pandey\cite{Mohan_Pandey_2023_restricted_sumset} studied  the following extended inverse theorems.
	\begin{theorem}\textup{\cite[Theorem 1.8]{Mohan_Pandey_2023_restricted_sumset}}\label{Extended inverse thm 2 fold}
		Let $k \geq 7$ be a positive integer. Let $A$ be a finite set of $k$ nonnegative integers with  $\min(A) =0$ and $d(A) = 1$. Then
		\begin{enumerate}
			\item[\upshape(1)]  $\left|2^{\wedge}A\right|= 2k-2$ if and only if $A = [0,k] \setminus \{x\}$, where $x \in \{ 1,2,k-2,k-1\}$;
			
			\item[\upshape(2)] for $k \geq 9$, $\left|2^{\wedge}A\right|= 2k-1$ if and only if  $A=[0,k+1] \setminus\{ x,y\}$, where $\{x,y\}$ is one of the sets  $\{1,2\}$, $\{k-1,k\}$, $\{2,3\}$, $\{k-2,k-1\}$,  $\{1,3\}$, $\{k-2,k\}$,  $\{1,4\}$, $\{k-3,k\}$, $\{1,k\}$, $\{1,k-1\}$, $\{2,k\}$, $\{2,k-1\}$, and $\{i,k+1\}$ where $3 \leq i \leq k-3$;
			
			\item[\upshape(3)] for $k \geq 11$, $\left|2^{\wedge}A\right|= 2k$ if and only if  $A=[0,k+2] \setminus\{ x,y,z\}$, where $\{x,y,z\}$ is one of the sets  $\{3,4,k+2\}$, $\{k-3,k-2, k+2\}$,   $\{2,4,k+2\}$,  $\{k-3,k-1, k+2\}$,  $\{2,5,k+2\}$,  $\{k-4,k-1, k+2\}$,   $\{1,2,3\}$, $\{k-1,k,k+1\}$,  $\{2,3,4\}$, $\{k-2,k-1,k\}$,  $\{1,2,4\}$, $\{k-2,k,k+1\}$, $\{1,2,k+1\}$, $\{1,k,k+1\}$, $\{1,3,4\}$, $\{k-2,k-1,k+1\}$,   $\{1,2,5\}$,  $\{k-3,k,k+1\}$,  $\{1,2,k\}$, $\{2,k,k+1\}$, $\{2,3,k\}$, $\{2,k-1,k\}$,   $\{1,2,6\}$, $\{k-4,k,k+1\}$, $\{2,3,k+1\}$, $\{1,k-1,k\}$, $\{1,3,5\}$, $\{k-3,k-1,k+1\}$, $\{1,3,k\}$, $\{2,k-1,k+1\}$, $\{1,3,k+1\}$, $\{1,k-1,k+1\}$,  $\{1,4,6\}$, $\{k-4,k-2,k+1\}$, $\{1,4,k\}$, $\{2,k-2,k+1\}$, $\{1,4,k+1\}$,  $\{1,k-2,k+1\}$, $\{i,j,k+2\}$ where $i \in \{1,2\}$ with $i+4 \leq j \leq k-2$, and $\{i,j, k+2\}$ where $3 \leq i  \leq j-4$ with $j \in \{k-1,k\} $. 
		\end{enumerate}
	\end{theorem}
	
	\begin{theorem}\textup{\cite[Theorem 1.9, Theorem 1.10]{Mohan_Pandey_2023_restricted_sumset}}\label{Mohan-pandey-extended  inverse thm 3^A}
		Let $k \geq 10$ be a positive integer.  Let $A$ be a finite set of $k$ integers with $\min(A) = 0$ and $d(A) = 1$. Then
		\begin{enumerate}
			\item [\upshape(1)]  $\left| 3^{\wedge}A \right| = 3k-7$ if and only if $A = [0,k] \backslash \{x\}$, where $x \in \{1,k-1\}$;
			
			\item [\upshape(2)] for $k \geq 12$, $\left| 3^{\wedge}A \right| = 3k-6$ if and only if $A = [0,k+1] \backslash \{x,y\}$, where $ \{x,y\}$ is one of the sets $\{2,k+1\}$, $\{3,k+1\}$, $\{k-3,k+1\}$,  $\{k-2,k+1\}$, $\{1,2\}$, $\{k-1,k\}$, and $\{1,k\}$;
			
			\item[\upshape(3)] for $k \geq 12$,  $\left| 4^{\wedge}A \right| = 4k-14$ if and only if $A = [0,k] \backslash \{x\}$, where $x \in \{1,k-1\}$.
			
		\end{enumerate} 
	\end{theorem}
	\noindent Based on the above theorems, they proposed the following conjecture.
	\begin{conjecture}\textup{\cite[Conjecture 4.1]{Mohan_Pandey_2023_restricted_sumset}}\label{Mohan_Pandey_conjecture}
		Let $k$ be a large positive integer and $h$ be a positive integer with $2 \leq h \leq k-2$. Let $A$ be a finite set of $k$ nonnegative integers with min$(A)=0$ and d$(A)=1$.
		\begin{enumerate}
			\item[\upshape(a)] If $|h^{\wedge}A|=hk-h^2+2$, then $A\subset [0,k]$.
			\item [\upshape(b)] If $|h^{\wedge}A|=hk-h^2+3$, then $A\subset [0,k+1]$.
			\item [\upshape(c)] If $|h^{\wedge}A|=hk-h^2+4$, then $A\subset [0,k+2]$.
		\end{enumerate}
	\end{conjecture} 
	\noindent In this article, we first prove the above conjecture in the form of following theorem.
	\begin{theorem}\label{Lemma 1}
		Let   $h\geq 3$ and $k$  be positive integers. Let  $A$ be a  finite set of $k$ integers such that $\min(A) =0$ and $d(A) =1$.
		\begin{enumerate}
			\item[\upshape(1)] If $k \geq 3h +1$ and $\left|h^{\wedge}A\right| = hk-h^{2}+2$, then $A \subseteq [0,k].$
			
			\item[\upshape(2)] If $k \geq 3h+3$ and $\left|h^{\wedge}A\right| = hk-h^{2}+3$, then $A \subseteq [0,k+1].$
			
			\item[\upshape(3)] If $k \geq 3h+4$ and $\left|h^{\wedge}A\right| = hk-h^{2}+4$, then $A \subseteq [0,k+2].$
		\end{enumerate}

	\end{theorem} 
	\noindent	In addition, we prove the following extended inverse results for $h^{\wedge}A$.
	\begin{theorem}\label{one element EIT}
		Let $h\geq 3$ and $k \geq 3h+1$ be positive integers. Let $A$ be a finite set of $k$  integers such that $\min(A) =0$, $d(A) =1$, and $|h^{\wedge}A|= hk-h^2+2$. Then $$A=[0,k] \setminus \{x\},$$ where $ x \in \{1,k-1\}.$
	\end{theorem}
	\begin{theorem}\label{two element EIT h=3}
		Let $k\geq 12$ be a positive integers. Let $A$ be a finite set of $k$  integers such that $\min(A) =0$, $d(A) =1$, and $|3^{\wedge}A|= 3k-6$. Then $$A= [0,k+1] \setminus \{x,y\},$$ where $\{x,y\}$ is one of the sets $\{1,2\},\{k-1,k\}, \{1,k\},\{2,k+1\}, \{3,k+1\},\{k-2,k+1\}$, and $\{k-3,k+1\}$.
	\end{theorem}
	
	\begin{theorem}\label{two element EIT}
		Let $h\geq 4$ and $k \geq 3h+3$ be positive integers. Let $A$ be a finite set of $k$  integers such that $\min(A) =0$, $d(A) =1$, and $|h^{\wedge}A|= hk-h^2+3$. Then $$A= [0,k+1] \setminus \{x,y\},$$ where $\{x,y\}$ is one of the sets $\{1,2\},\{k-1,k\}, \{1,k\},\{2,k+1\}$, and $\{k-2,k+1\}$.
	\end{theorem}

	\begin{theorem}\label{3 element EIT h=3}
		Let $k \geq 13$ be a positive integer. Let $A$ be a set of $k$  integers such that $\min(A) =0$, $d(A) =1$, and $|3^{\wedge}A|= 3k-5$. Then $$A= [0,k+2] \setminus \{x,y,z\},$$ where $\{x,y,z\}$ is one of the sets $\{1,2,3\}, \{k-1,k,k+1\}, \{1,2,k+1\}, \{1,k,k+1\}, \{1,3,k+2\}, \{k-2,k,k+2\}, \{1,4,k+2\}, \{k-3,k,k+2\}, \{2,k,k+2\}, \{1,k-1,k+2\}, \{1,k-2,k+2\}, \{3,k,k+2\}$, and $\{r,k+1,k+2\}$ where $4 \leq r \leq k-4.$
	\end{theorem}
	
	\begin{theorem}\label{3 element EIT h=4}
		Let  $k \geq 16$ be a positive integer. Let $A$ be a finite set of $k$ integers such that $\min(A) =0$, $d(A) =1$, and $|4^{\wedge}A|= 4k-12$. Then $$A= [0,k+2] \setminus \{x,y,z\},$$ where $\{x,y,z\}$ is one of the sets $\{1,2,3\}, \{k-1,k,k+1\}, \{1,2,k+1\}, \{1,k,k+1\}, \{1,3,k+2\}, \{k-2,k,k+2\}, \{2,k,k+2\}, \{1,k-1,k+2\}, \{3,k+1,k+2\},\{k-3,k+1,k+2\},\{4,k+1,k+2\}$, and  $\{k-4,k+1,k+2\}.$
	\end{theorem}
	
	\begin{theorem}\label{3 element EIT gen}
		Let $h\geq 5$ and $k \geq 3h+4$ be positive integers. Let $A$ be a finite set of $k$  integers such that $\min(A) =0$, $d(A) =1$, and $|h^{\wedge}A|= hk-h^2+4$. Then $$A= [0,k+2] \setminus \{x,y,z\},$$ where $\{x,y,z\}$ is one of the sets $\{1,2,3\}, \{k-1,k,k+1\}, \{1,2,k+1\}, \{1,k,k+1\}, \{1,3,k+2\}, \{k-2,k,k+2\}, \{2,k,k+2\}, \{1,k-1,k+2\}, \{3,k+1,k+2\}$, and $\{k-3,k+1,k+2\}.$
	\end{theorem}
	\noindent Proof of all the above Theorems given in Section \ref{Proofs}. Due to \eqref{Eq-1}, above results also true for $(k-h)^{\wedge}A$.

	\section{Proofs}\label{Proofs}
	\begin{proof}[Proof of Theorem \ref{Lemma 1}]
		Let   $h\geq 3$ and $k$  be positive integers with $k \geq 3h +1$.	Let $A = \{a_{0}, a_{1}, \ldots , a_{k-1}\}$ be a set of $k$ integers such that $$0=a_{0}<a_{1}< \cdots < a_{k-1}$$ with $d(A) = 1$ and $\left|h^{\wedge}A\right| = hk-h^{2}+2$.  Let $A_{1} = \{a_{0}, a_{h-1}, a_{h}, \ldots, a_{k-1}\}$.
		
		\textbf{Claim:} $d(A_{1}) =1$. 
		
		If $d(A_{1}) >1$, then there exists a positive integer $t$ such that 
		\begin{equation}\label{Theorem-Equation1}
			1  \leq t < h-1,
		\end{equation} $d(A_{2}) =1$ and $d(A_{3}) >1$, where $A_{2} = \{a_{0}, a_{t}, a_{t+1}, \ldots, a_{k-1}\}$ and $A_{3} = A_{2}\setminus \{a_{t}\}$. To get a contradiction,  we show that $\left|h
		^{\wedge}A\right| > hk-h^{2}+2$.
		
		Since $d(A_{3})>1$ and $d(A_{2}) =1$, the sets $(h-1)^{\wedge}A_{3} + a_{t}$ and $h^{\wedge}A_{3}$ are disjoint subsets of $h^{\wedge}A$. This, together with \eqref{Theorem-Equation1} and  Theorem \ref{Direct Thm Nathanson restricted}, we get that \begin{align*}
			\left|h^{\wedge}A\right| &\geq \left|(h-1)^{\wedge}A_{3} + a_{t}\right| + \left|h^{\wedge}A_{3}\right|\\
			& =  \left|(h-1)^{\wedge}A_{3}\right| + \left|h^{\wedge}A_{3}\right|\\
			& \geq (h-1)(k-t)-(h-1)^{2}+1 + h(k-t)-h^{2}+1\\
			& =2hk-2ht-k+t-2h^{2}+2h +1\\
			&\geq hk  +  (h-1)k-2h(h-2) +1-2h^{2} +2h+1\\
			&\geq hk+ (h-1)(3h+1) -4h^{2}+6h+2\\
			& =hk-h^{2}+4h+1>hk-h^{2}+4,
		\end{align*}
		which is a contradiction. Thus, $d(A_{1})=1.$

		Now,  let $A _{4} = A \setminus\{a_{0},a_{k-2}, a_{k-1}\}$. Then, clearly \[(a_{1}+a_{2}+\ldots+a_{h-2}+2^{\wedge}A_{1}) \cup ((h-2)^{\wedge}A_{4}+a_{k-2}+a_{k-1}) \subseteq h^{\wedge}A\] and \[(a_{1}+a_{2}+\ldots+a_{h-2}+2^{\wedge}A_{1}) \cap (a_{k-2}+a_{k-1}+(h-2)^{\wedge}A_{4}) = \{a_{1}+a_{2}+\ldots+a_{h-2}+a_{k-2}+a_{k-1}\}. \] Therefore, by Theorem \ref{Direct Thm Nathanson restricted}, we get that
		\begin{align*}
			hk - h^{2} + 2 = |h^ {\wedge}A| &\geq |a_{1}+a_{2}+\cdots+a_{h-2}+2^ {\wedge}A_{1}| + |a_{k-2}+a_{k-1}+(h-2)^{\wedge}A_{4}|-1\\ & = |2^{\wedge}A_{1}| + |(h-2)^{\wedge}A_{4}|-1 \\
			& \geq |2^{\wedge}A_{1}|+(h-2)(k-3)-(h-2)^{2}+1-1.
		\end{align*} This gives that  $$|2^{\wedge}A_1| \leq 2k-h.$$
		Since $k \geq 3h+1$, we have \[|2^{\wedge}A_1|  \leq 2k-h  < 2.6(k-h+2)-6.\]
		Therefore by Theorem \ref{Lev_restricted_thm}, we get that  $a_{k-1} \leq 2(k-h+2)-5$. Again, using  Theorem \ref{Lev_restricted_thm}, we obtain \[2k-h \geq \left|2^{\wedge}A_{1}\right| \geq a_{k-1} + k-h+2-2.\]
		This gives $a_{k-1} \leq k$. Hence $A \subseteq  [0,k]$. This proves $(1)$. Using a similar argument, we can prove $(2)$ and $(3)$. 
	\end{proof}
	Now, we prove auxillary propositions to prove Theorem \ref{one element EIT} -
	Theorem \ref{3 element EIT gen}. 
	\begin{prop}\label{one element EIT prop}
		Let $ h \geq 3$ and $k \geq 3h+1 $ be positive integers. Let  \[A=[0,k] \setminus \{x\}, ~\text{where}~ x \in [1, k-1].\]
		\begin{enumerate}
			\item [\upshape(i)] If $x \in [1, h-1]$, then $|h^{\wedge}A|= hk-h^2+x+1$.
			
			\item[\upshape (ii)] If $x \in [k-h+1,k-1]$, then $|h^{\wedge}A|=
			(h+1)k-h^2-x+1$.
			
			\item [\upshape(iii)]  If $x \in \{h,k-h\}$, then  $|h^{\wedge}A|= hk-h^2+h $.      
			\item [\upshape(iv)]  If $x \in [h+1, k-h-1]$, then $|h^{\wedge}A|=hk-h^2+h+1$.
			
		\end{enumerate}    
	\end{prop}
	\begin{proof}
		\begin{enumerate}
			\item [\upshape(i).]
			If $x \in [1, h-1]$, then $A=[0,x-1] \cup [x+1,k]$. It is easy to see that \begin{center} 
				$h^{\wedge}A =
				\bigg[\dfrac{h(h+1)}{2}-x, hk - \dfrac{h(h-1)}{2}\bigg].$
			\end{center} 
			Therefore, we have $|h^{\wedge}A|=hk-h^2+x+1.$
			\item [\upshape (ii)] If $ x \in [k-h+1,k-1]$, then
			\begin{align*}
				A& =[0,x-1] \cup [x+1,k]\\ & = k- ([0,k-x-1] \cup [k-x+1,k])
				\\ & = k- ([0,y-1] \cup [y+1,k]),
			\end{align*}
			where $y=k-x \in [1,h-1].$
			Since the cardinality of $h^{\wedge}A$ is translation and dilation invariant, we have 
			\begin{center}
				$|h^{\wedge}A| = |h^{\wedge}([0,y-1] \cup  [y+1,k])|, ~ \text{where}~ y \in [1,h-1].$ 
			\end{center} Therefore, from the previous case, we have  \begin{align*}
				|h^{\wedge}A| &=hk-h^2+y+1 \\&=hk-h^2+(k-x)+1\\&= (h+1)k -h^2 -x+1.
			\end{align*}
			\item 	[\upshape(iii)]  
			\begin{enumerate}
				\item If $x=h$,  then $A=[0,h-1] \cup [h+1,k]$. It is easy to observe that $$\dfrac{h(h- 1)}{2}+1 \notin h^{\wedge}A$$ and 
				\begin{center}$h^{\wedge}A = \left\{\dfrac{h(h-1)}{2}\right\} \bigcup   \left[\dfrac{h(h- 1)}{2}+2, hk-\dfrac{h(h-1)}{2}\right].$
				\end{center}
				Therefore,  $ |h^{\wedge}A| = hk-h^2+h.$
				\item 	If $x=k-h$, then 
				$A=[0,k-(h+1)] \cup [k-(h-1),k]= k-([0,h-1] \cup [h+1,k]).$
				Since the cardinality of $h^{\wedge}A$ is translation and dilation invariant, we have 
				$$|h^{\wedge}A|= |h^{\wedge}([0,h-1] \cup [h+1,k])|.$$
				Therefore, in both the cases, we have  $ |h^{\wedge}A| = hk-h^2+h.$
			\end{enumerate}

			\item 	[\upshape(iv)]   If $x \in [h+1,k-h-1]$, then $A=[0,x-1] \cup [x+1,k]$. It follows that 
			\begin{center}
				$h^{\wedge}A =\bigg[ \dfrac{h(h-1)}{2}, hk-\dfrac{h(h-1)}{2}\bigg].$
			\end{center} Therefore, $ |h^{\wedge}A| = hk-h^2+h+1.$
		\end{enumerate}
	\end{proof}
	\noindent In  the following proposition we compute the cardinality of $h^{\wedge}A$ when $A=[0,k+1] \setminus \{x,x+1\}, ~\text{where}~ x\in [1, k-1].$
	\begin{prop}
		Let $ h \geq 3$ and $k \geq 3h+3 $ be positive  integers. Let  \[A=[0,k+1] \setminus \{x,x+1\}, ~\text{where}~ x\in [1, k-1].\]
		\begin{enumerate}
			\item [\upshape(i)] If $x \in [1, h-1] $, then  $|h^{\wedge}A|= hk-h^2+2x+1$.
			
			\item[\upshape(ii)] If $x \in [k-h+1,k-1]$, then $|h^{\wedge}A|=	(h+2)k-h^2-2x+1$. 
			
			\item [\upshape(iii)]    If $x \in \{h,k-h\}$, then $|h^{\wedge}A|= hk-h^2+2h-1$.
			\item [\upshape(iv)]  If $x \in [h+1, k-h-1]$, then $|h^{\wedge}A|=hk-h^2+2h+1$.
			
		\end{enumerate}    
	\end{prop}
	
	\begin{proof}
		\begin{enumerate}
			\item [\upshape(i)]
			If $x \in [1, h-1]$, then $A=[0,x-1] \cup [x+2,k+1]$. It is easy to see that 
			$$h^{\wedge}A =
			\bigg[\dfrac{(h+1)(h+2)}{2}-(2x+1), hk - \dfrac{h^2-3h}{2}\bigg].$$
			Therefore,  $|h^{\wedge}A|=hk-h^2+2x+1 .$
			\item  [\upshape (ii) ] If $ x \in [k-h+1,k-1]$, then
			\begin{align*}A &=[0,x-1] \cup [x+2,k+1]
				\\&= (k+1)- \big([0,k-x-1] \cup  [k-x+2,k+1]\big) \\&= (k+1)- \big([0,y-1] \cup  [y+2,k+1]\big), 
			\end{align*} where  $y =k-x\in [1,h-1].$
			Since the cardinality of $h^{\wedge}A$ is translation and dilation invariant, we have 
			$$|h^{\wedge}A| = |h^{\wedge}([0,y-1]  \cup [y+2,k+1])|, ~ \text{where}~ y \in [1,h-1].$$
			Therefore, from the previous case, we have  \begin{align*}
				|h^{\wedge}A| &=hk-h^2+2y+1\\ &=hk-h^2+2(k-x)+1\\&=(h+2)k -h^2 -2x +1.
			\end{align*}

			\item 	[\upshape(iii)]  \begin{enumerate}
				\item 
				If $x=h$, then $A=[0,h-1] \cup [h+2,k+1]$. It is easy to observe that
				$\dfrac{h(h- 1)}{2}+1$, $\dfrac{h(h- 1)}{2}+2$ are not in $h^{\wedge}A$,
				and   
				\begin{center}$h^{\wedge}A =\bigg\{ \dfrac{h(h- 1)}{2}\bigg\} \bigcup \bigg[\dfrac{h(h- 1)}{2}+3, hk-\dfrac{h(h-3)}{2}\bigg] .$\end{center}
				Therefore, $ |h^{\wedge}A| =  hk-h^2+2h-1.$

				\item 	If $x=k-h$, then we have \begin{align*}
					A& =[0,k-h-1] \cup [k-(h-2),k+1]\\
					& = (k+1)-\big([0,h-1] \cup [h+2,k+1]\big).\end{align*} Since the cardinality of $h^{\wedge}A$ is translation and dilation invariant, we have \[|h^{\wedge}A| = |h^{\wedge}\big([0,h-1] \cup [h+2,k+1]\big)|.\] Therefore, in both the cases, we have \[ |h^{\wedge}A| = hk-h^2+2h-1.\]
			\end{enumerate}
			
			\item 	[\upshape(iv)]   If $x \in [h+1,k-h-1]$, then $A=[0,x-1] \cup [x+2,k+1]$. It is evident that $$h^{\wedge}A =\bigg[ \dfrac{h(h-1)}{2}, hk-\dfrac{h^2-3h}{2}\bigg].$$ Therefore, $ |h^{\wedge}A| = hk-h^2+2h+1.$		
		\end{enumerate}
	\end{proof}
	
	\noindent	In the following proposition, we compute the cardinality of $h^{\wedge}A$ when $A=[0,k+1] \setminus \{x,x+2\}, ~\text{where}~ x\in [1,k-2].$
	
	\begin{prop}
		Let $ h \geq 3$ and $k \geq 3h+3$ be positive  integers. Let  \[A=[0,k+1] \setminus \{x,x+2\}, ~\text{where}~ x\in [1,k-2].\]
		\begin{enumerate}
			\item [\upshape(i)] If $x \in [1, h-2]$, then $|h^{\wedge}A|=hk-h^2+2x+2$.
			
			\item[\upshape(ii)] If $x \in [k-h+1,k-2]$, then $|h^{\wedge}A|= (h+2)k-h^2-2x$.
			
			\item [\upshape(iii)] If $x \in \{h-1,k-h\}$, then  $|h^{\wedge}A|=hk-h^2+2h-1$.
			\item [\upshape(iv)]  If $x \in \{h,k-h-1\}$, then
			$|h^{\wedge}A|= hk-h^2+2h $.      
			\item [\upshape(v)]  If $x \in [h+1, k-h-2]$, then $|h^{\wedge}A|=hk-h^2+2h+1$.
			
		\end{enumerate}    
	\end{prop}
	
	\begin{proof}
		\begin{enumerate}
			\item [\upshape(i)]
			If $x \in [1, h-2]$, then $A=[0,x-1] \cup \{x+1\} \cup [x+3,k+1]$.
			Clearly, \begin{center}
				$h^{\wedge}A =
				\bigg[\dfrac{(h+1)(h+2)}{2}-(2x+2), hk - \dfrac{h^2-3h}{2}\bigg].$
			\end{center}
			Therefore, $|h^{\wedge}A|=	hk-h^2+2x+2 .$
			\item[\upshape(ii)]  If $ x \in [k-h+1,k-2]$, then
			\begin{align*}
				A &=[0,x-1] \cup\{x+1\} \cup [x+3,k+1]\\
				&=(k+1)- ([0,k-x-2] \cup \{k-x\} \cup  [k-x+2,k+1])\\
				&= (k+1)- ([0,y-1] \cup \{y+1\} \cup  [y+3,k+1]),
			\end{align*}   
			where $y=k-x-1 \in [1,h-2].$ Since the cardinality of $h^{\wedge}A$ is translation and dilation invariant, we have 
			\[|h^{\wedge}A| = |h^{\wedge}([0,y-1] \cup \{y+1\} \cup  [y+3,k+1]),~\text{where}~ y \in [1,h-2].\] Therefore,  from the previous case, we have \begin{align*}
				|h^{\wedge}A|&=hk-h^2+2y+2 \\&=hk-h^2+2(k-x-1)+2 \\& =(h+2)k-h^2-2x.
			\end{align*}
			
			\vspace{.1cm}
			\item 	[\upshape(iii)]  \begin{enumerate}
				\item 
				If $x=h-1$,  then $A=[0,h-2] \cup \{h\} \cup [h+2,k+1]$. It is easy to see that 
				$\dfrac{(h-1)(h-2)}{2}+h+1 \notin h^{\wedge}A$ and   
				\begin{center}
					$h^{\wedge}A =\bigg\{\dfrac{(h-1)(h-2)}{2}+h \bigg \} \bigcup \bigg[ \dfrac{(h-2)(h- 1)}{2}+h+2, hk-\dfrac{h(h-3)}{2}\bigg] .$
				\end{center}
				Therefore,  $|h^{\wedge}A| = hk-h^2+2h-1.$
				
				\item 	If $x=k-h$, then  \begin{align*}
					A &=[0,k-h-1] \cup \{k-h+1\} \cup [k-(h-3),  k+1]\\ & = (k+1)-([0,h-2] \cup \{h\} \cup [h+2,k+1]).
				\end{align*} 
				Since the cardinality of $h^{\wedge}A$ is translation and dilation invariant, we have 
				$$|h^{\wedge}A|= |h^{\wedge}([0,h-2] \cup \{h\} \cup  [h+2,k+1])|.$$   Therefore, in both the cases, we have   $|h^{\wedge}A| = hk-h^2+2h-1.$
			\end{enumerate}
			
			\item 	[\upshape(iv)]  \begin{enumerate}
				\item 
				If $x=h$,  then $A=[0,h-1] \cup \{h+1\} \cup [h+3,k+1]$. It is easy to observe that $\dfrac{h(h- 1)}{2}+1 \notin h^{\wedge}A$ and  \begin{center}$h^{\wedge}A =\bigg\{\dfrac{h(h-1)}{2} \bigg\} \bigcup \bigg[ \dfrac{h(h- 1)}{2}+2, hk-\dfrac{h(h-3)}{2}\bigg].$\end{center}
				Therefore,   $|h^{\wedge}A| = hk-h^2+2h.$
				\item 	If $x=k-h-1$, then we have \begin{align*}
					A&=[0,k-h-2] \cup\{k-h\} \cup [k-(h-2),k+1]\\
					&= (k+1)-([0,h-1] \cup \{h+1\} \cup [h+3,k+1]).
				\end{align*} Since the cardinality of $h^{\wedge}A$ is translation and dilation invariant, we have 
				\begin{align*}|h^{\wedge}A|= |h^{\wedge}([0,h-1] \cup \{h+1\} \cup [h+3,k+1])|. \end{align*}  Therefore, in both the cases, we have  $$ |h^{\wedge}A| = hk-h^2+2h.$$
			\end{enumerate}
			
			\item 	[\upshape(v)]   If $x \in [h+1,k-h-2]$, then $A=[0,x-1] \cup \{x+1\} \cup [x+3,k+1]$. Clearly, \begin{center}$h^{\wedge}A =\bigg[ \dfrac{h(h-1)}{2}, hk-\dfrac{h^2-3h}{2}\bigg].$
			\end{center} Therefore, $ |h^{\wedge}A| = hk-h^2+2h+1.$
		\end{enumerate}
	\end{proof}
	
	\noindent In the following proposition we compute the cardinality of $h^{\wedge}A$ when $A=[0,k+1] \setminus \{x,k\}$, where $x\in [1, k-3].$

	\begin{prop} \label{prop1.4}
		Let $ h \geq 3$ and $k \geq 3h+3 $ be positive integers. Let  \[A=[0,k+1] \setminus \{x,k\}, ~\text{where}~ x\in [1, k-3].\]
		\begin{enumerate}
			\item [\upshape(i)] If $x \in [1, h-1]$, then  $|h^{\wedge}A|= hk-h^2+x+2.$
			
			\item [\upshape(ii)]    If $x \in \{h,k-h\}$, then $|h^{\wedge}A|= hk-h^2+h+1.$
			
			\item [\upshape(iii)]  If $x \in [h+1, k-h-1]$, then $|h^{\wedge}A|=hk-h^2+h+2.$

			\item [\upshape(iv)]  If $x \in [k-(h-1), k-3]$, then $|h^{\wedge}A|=(h+1)k-h^2-x+2.$

		\end{enumerate} 
	\end{prop}   
	
	\begin{proof}
		\begin{enumerate}
			\item [\upshape(i)]
			If $x \in [1, h-1]$, then $A=[0,x-1]\cup [x+1,k-1] \cup \{k+1\}$. It is evident that \begin{center}
				$h^{\wedge}A =
				\bigg[\dfrac{h(h+1)}{2}-x, hk+1 - \dfrac{h(h-1)}{2}\bigg].$
			\end{center} 
			Therefore, $|h^{\wedge}A|=	hk-h^2+x+2.$
				
			\item 	[\upshape(ii)]
			\begin{enumerate}
				\item 
				If $x =h$, then $A=[0,h-1]\cup [h+1,k-1] \cup \{k+1\}$.  It is easy to observe that $\dfrac{h(h-1)}{2} +1 \notin h^{\wedge}A$ and
				\begin{center}
					$h^{\wedge}A =
					\bigg\{\dfrac{h(h-1)}{2}\bigg\} \bigcup	\bigg[\dfrac{h(h-1)}{2}+2, hk+1 - \dfrac{h(h-1)}{2}\bigg].$
				\end{center} 
			Therefore, we have $|h^{\wedge}A|=hk-h^2+h+1.$

				\item If $x =k-h$, then $A=[0,k-h-1]\cup [k-h+1,k-1] \cup \{k+1\}$. It follows that  $hk - \dfrac{h(h-1)}{2} \notin h^{\wedge}A$ and
				\begin{center}
					$h^{\wedge}A =
					\bigg[\dfrac{h(h-1)}{2}, hk - \dfrac{h(h-1)}{2} -1\bigg] \bigcup \bigg\{ hk - \dfrac{h(h-1)}{2} +1\bigg\}.$
				\end{center}
				Therefore, we have $|h^{\wedge}A|=hk-h^2+h+1.$
			\end{enumerate}

			\item 	[\upshape(iii)]   If $x \in [h+1,k-h-1]$, then $A=[0,x-1] \cup [x+1,k-1] \cup \{k+1\}$. 
			Clearly,  \begin{center}
				$h^{\wedge}A =\bigg[ \dfrac{h(h-1)}{2}, hk+1-\dfrac{h(h-1)}{2}\bigg].$
			\end{center} Therefore, 
			$|h^{\wedge}A| = hk-h^2+h+2.$
			
			\item [\upshape(iv)]
			If $x \in [k-(h-1), k-3]$, then $A=[0,x-1]\cup [x+1,k-1] \cup \{k+1\}$.
			It is easy to  observe that 
			\begin{center}
				$h^{\wedge}A =
				\bigg[\dfrac{h(h-1)}{2}, (h+1)k+1- \dfrac{h(h+1)}{2}-x\bigg].$
			\end{center} 
			Therefore, $|h^{\wedge}A|=	hk-h^2+k+2-x.$
		\end{enumerate}
	
	\end{proof}
	
	\noindent In the following proposition, we compute the cardinality of $h^{\wedge}A$ when $A=[0,k+1] \setminus \{1,y\}$, where $y\in [4, k].$ Since $h^{\wedge}A$ is translation and dilation invariant, we have $|h^{\wedge}A|=|h^{\wedge}(k+1-A)|=||h^{\wedge}([0,k+1] \setminus \{x,k\})|$, where $x\in [1, k-3].$ Therefore,  the proof of the following proposition follows from Proposition \ref{prop1.4}.
	
	\begin{prop}
		Let $h \geq 3$ and $k \geq 3h-3$ be positive integers. Let $$A=[0,k+1] \setminus \{1,y\}, ~ \text{where}~ y\in [4, k].$$
		\begin{enumerate}
			\item[ \upshape (i)] If $y \in [k-(h-2),k]$, then $|h^{\wedge}A|= (h+1)k-h^2-y+3.$
			\item[ \upshape (ii)] If $y\in \{h+1,k-(h-1)\}$, then $|h^{\wedge}A|=hk-h^2+h+1.$
			
			\item[ \upshape (iii)] If $y \in [h+2,k-h]$, then $|h^{\wedge}A|=hk-h^2+h+2.$
			\item[ \upshape (iv)] If $y \in [4,h]$, then $|h^{\wedge}A|= (h-1)k-h^2+y+1.$
		\end{enumerate}
		
	\end{prop}

	\noindent In the following proposition, we compute the cardinality of $h^{\wedge}A$ when $A=[0,k+1] \setminus \{x,k-1\},$ where $x\in [1, k-4].$
	
	\begin{prop}\label{prop1.6}
		Let $h \geq 3$ and $k \geq 3h+3$ be positive integers. Let $$A=[0,k+1] \setminus \{x,k-1\}, ~ \text{where}~ x\in [1, k-4].$$
		\begin{enumerate}
			\item[ \upshape (i)] If $x \in [1,h-1]$, then $|h^{\wedge}A|= hk-h^2+x+3.$
			\item[ \upshape (ii)] If $x\in \{h,k-h\}$, then $|h^{\wedge}A|=hk-h^2+h+2.$
			
			\item[ \upshape (iii)] If $x \in [h+1, k-h-2]$, then $|h^{\wedge}A|= hk-h^2+h+3.$
			
			\item[ \upshape (iv)] If $x=k-h-1$, then $|h^{\wedge}A|=
			\begin{cases}
				hk-h^2+h+3, & \text{ if } h \geq 4;\\
				3k-4, & \text{ if } h=3.
			\end{cases}$
			
			\item[\upshape (v)] If $x \in [k-h+1,k-4]$, then $|h^{\wedge}A|= 	(h+1)k-h^2-x+3.$
		\end{enumerate}
		
	\end{prop}
	\begin{proof}
		
		\begin{enumerate}
			\item [\upshape(i)]
			
			If $x \in [1, h-1]$, then $A=[0,x-1]\cup [x+1,k-2] \cup \{k,k+1\}$. It is apparent that \begin{center}
				$h^{\wedge}A =
				\bigg[\dfrac{h(h+1)}{2}-x, hk+2 - \dfrac{h(h-1)}{2}\bigg].$
			\end{center}
			Therefore,
				$|h^{\wedge}A|=	hk-h^2+x+3.$
				
			\item 	[\upshape(ii)]
			\begin{enumerate}
				\item 
				If $x =h$, then $A=[0,h-1]\cup [h+1,k-2] \cup \{k,k+1\}$.  Note that $\dfrac{h(h-1)}{2} +1 \notin h^{\wedge}A$ and \begin{center}
					$h^{\wedge}A = \bigg\{\dfrac{h(h-1)}{2} \bigg\}\bigcup
					\bigg[\dfrac{h(h-1)}{2}+2, hk+2 - \dfrac{h(h-1)}{2}\bigg].$
				\end{center}
				Therefore,	$|h^{\wedge}A|=hk-h^2+h+2.$

				\item If $x =k-h$, then $A=[0,k-h-1]\cup [k-h+1,k-2] \cup \{k,k+1\}$.  Note that 
				$hk+1 - \dfrac{h(h-1)}{2}  \notin h^{\wedge}A$. Clearly, \begin{center}
					$h^{\wedge}A =
					\bigg[\dfrac{h(h-1)}{2}, hk - \dfrac{h(h-1)}{2}\bigg] \bigcup \bigg\{ hk +2 - \dfrac{h(h-1)}{2}\bigg\}.$
				\end{center} 
				Therefore,
					$|h^{\wedge}A|=hk-h^2+h+2.$
			\end{enumerate}

			\item 	[\upshape(iii)]   If $x \in [h+1,k-h-2]$, then $A=[0,x-1] \cup [x+1,k-2] \cup \{k,k+1\}$. It is easy to see that \begin{center}
				$h^{\wedge}A =\bigg[ \dfrac{h(h-1)}{2}, hk+2-\dfrac{h(h-1)}{2}\bigg].$
			\end{center} Therefore,  $|h^{\wedge}A| = hk-h^2+h+3.$
			
			\item 	[\upshape(iv)] 
			\begin{enumerate}
				\item 
				If $h=3$ and $x=k-4$, then $A=[0,k-5] \cup \{k-3,k-2,k,k+1\}$. Note that $3k-3 \notin 3^{\wedge}A$ and $3^{\wedge}A=[3,3k-4] \cup \{3k-2,3k-1\}$. Therefore, $|3^{\wedge}A|=3k-4.$
				
				\item If $h\geq 4$ and $x=k-h-1$, then $A=[0,k-h-2] \cup [k-h,k-2] \cup \{k,k+1\}$. It follows that 
				\begin{center}
					$h^{\wedge}A =\bigg[ \dfrac{h(h-1)}{2}, hk+2-\dfrac{h(h-1)}{2}\bigg].$
				\end{center} Therefore, 
				$|h^{\wedge}A| = hk-h^2+h+3.$
			\end{enumerate}

			\item [\upshape(v)]
			
			If $x \in [k-(h-1), k-4]$, then $A=[0,x-1]\cup [x+1,k-2] \cup \{k,k+1\}$. One can see easily 
			\begin{center}
				$h^{\wedge}A =
				\bigg[\dfrac{h(h-1)}{2}, (h+1)k+2- \dfrac{h(h+1)}{2}-x\bigg].$
			\end{center} 
			Therefore, 
				$|h^{\wedge}A|=	hk-h^2+k+3-x.$ 
		\end{enumerate}
	\end{proof}
	
	\noindent In the following proposition, we compute the cardinality of $h^{\wedge}A$ when $A=[0,k+1] \setminus \{2,y\}$, where $y\in [5, k].$ Since the cardinality of  $h^{\wedge}A$ is translation and dilation invariant, we have $|h^{\wedge}A| = |h^{\wedge}(k+1-A)| = |h^{\wedge}[0,k+1] \setminus \{x,k-1\},$ where $x\in [1, k-4].$ Therefore, the proof of the following proposition follows from the Proposition \ref{prop1.6}.
	
	\begin{prop}
		Let $h \geq 3$ and $k \geq 3h+3$ be positive integers. Let $$A=[0,k+1] \setminus \{2,y\},~ \text{where}~ y\in [5, k].$$
		\begin{enumerate}
			\item[ \upshape (i)] If $y \in [k-(h-2),k]$, then $|h^{\wedge}A|= hk-h^2+k-y+4.$
			\item[ \upshape (ii)] If $y\in \{h+1,k-(h-1)\}$, then $|h^{\wedge}A|=hk-h^2+h+2.$
			\item[ \upshape (iii)] If $y \in [h+3,k-h]$, then $|h^{\wedge}A|=hk-h^2+h+3.$
			\item[ \upshape (iv)] If $x=h+2$, then $|h^{\wedge}A|=
			\begin{cases}
				hk-h^2+h+3, & \text{ if } h \geq 4;\\
				3k-4, & \text{ if } h=3.
			\end{cases}$
			
			\item[ \upshape (v)] If $y \in [5,h]$, then $|h^{\wedge}A|= hk-h^2+y+2.$
		\end{enumerate}
		
	\end{prop}
	
	\noindent Proposition \ref{gen {x,y} deletion prop} demonstrates the cardinality of $h^{\wedge}A$ when $A=[0,k+1] \setminus \{x,y\},~ \text{where} ~ x,y \in [3,k-2] ~ \text{and} ~ y-x \geq 3.$ 
	\begin{prop}\label{gen {x,y} deletion prop}
		Let $h \geq 3$ and $k \geq 3h+3$ be positive  integers. Let $$A=[0,k+1] \setminus \{x,y\},~ \text{where} ~ x,y \in [3,k-2] ~ \text{and} ~ y-x \geq 3.$$
		\begin{enumerate}
			\item[\upshape (i)] If $h\geq 6$, $x \in [3,h-3]$ and $y \in [6,h]$,  then $|h^{\wedge}A|= hk-h^2+x+y.$ 
			\item[\upshape (ii)] If $h\geq 5$ and $\{x,y\}=\{r,h+1\}$ where $r \in [3,h-2]$,  then $|h^{\wedge}A|= hk-h(h-1)+r.$
			\item [\upshape (iii)] If $h\geq 4$, $x \in [3,h-1]$ and $y \in [h+2,k-h]$, then $|h^{\wedge}A|=hk-h(h-1)+x+1.$
			\item[\upshape (iv)] If $h\geq 4$ and $\{x,y\}=\{r,k-h+1\}$ where $r \in [3,h-1]$, then $|h^{\wedge}A|= hk-h(h-1)+r.$
			
			\item[\upshape (v)] If $h\geq 4$, $x\in [3,h-1]$ and $y \in [k-(h-2),k-2]$, then $|h^{\wedge}A|=(h+1)k-h^2+x-y+2.$
			\item[\upshape (vi)] If $\{x,y\}$ is one of the sets $\{h,s\}$ with $s \in [h+3,k-h]$, $\{r,k-h+1\}$ with $r \in [h+1, k-h-2]$, then 
			$|h^{\wedge}A|= hk-h^2+2h.$
			\item[\upshape (vii)] If   $\{x,y\} = \{h,k-h+1\}$, then $|h^{\wedge}A|= hk-h^2+2h-1.$
			\item [\upshape (viii)] If $h\geq 4$ and $\{x,y\}=\{h,s\}$ where $s \in [k-h+2,k-2]$, then $|h^{\wedge}A|= (h+1)k-h(h-1) -s +1.$
			\item[\upshape (ix)] If $x \in [h+1,k-h-3]$ and $y \in [h+4,k-h]$, then $|h^{\wedge}A|= hk-h^2+2h+1$.
			\item [\upshape (x)] If $h\geq 4$, $x \in [h+1,k-h-1)]$ and $y \in [k-h+2,k-2]$, then $|h^{\wedge}A|=(h+1)k-h(h-1)-y+2.$
			\item [\upshape (xi)] If $h\geq 5$ and $\{x,y\}=\{k-h,s\}$ where $ s \in [k-h+3,k-2]$, then $|h^{\wedge}A|= (h+1)k-h(h-1) -s +1.$
			\item[ \upshape (xii)] If $h\geq 6$, $x \in [k-h+1,k-5]$, and $y \in [k-h+4,k-2]$, then $|h^{\wedge}A|= (h+2)k-h^2-x-y+2. $
		\end{enumerate}
	\end{prop}
	\begin{proof}
		\begin{enumerate}
			\item [(i)]
			If $x \in [3, h-3]$ and $y \in [6,h]$, then $A=[0,x-1]\cup [x+1,y-1] \cup  [y+1,k+1]$. It is easy to see that 
			\begin{center}
				$h^{\wedge}A =
				\bigg[\dfrac{(h+1)(h+2)}{2}-(x+y), hk - \dfrac{h(h-3)}{2}\bigg].$
			\end{center} 
			Therefore, we have 
			$|h^{\wedge}A|=	hk-h^2+x+y.$
			\item [(ii)]
			If $\{x,y\}=\{r,h+1\}$, where $r \in [3,h-2]$, then 
			\vspace{.1cm}
			\begin{center}
				$A=[0,r-1]\cup [r+1,h] \cup [h+2,k+1].$
			\end{center}  	\vspace{.1cm}
			Note that $\dfrac{h(h+1)}{2}-r+1 \notin h^{\wedge}A$ and
			\begin{center}
				$h^{\wedge}A=\left\{\dfrac{h(h+1)}{2}-r\right\} \bigcup \left[\dfrac{h(h+1)}{2}-r+2, hk - \dfrac{h(h-3)}{2}\right].$
			\end{center}
			Therefore, $|h^{\wedge}A|=hk-h^2+h+r.$	
			
			\item[\upshape(iii)] If $x \in [3,h-1]$ and $y \in [h+2,k-h]$, then $A=[0,x-1] \cup [x+1,y-1] \cup [y+1,k+1]$. It is easy to see that 
			\begin{center}
				$h^{\wedge}A= \left[\dfrac{h(h+1)}{2}-x, hk-\dfrac{h(h-3)}{2}\right]$.
			\end{center}
			Therefore, we get that 
			$|h^{\wedge}A|= hk-h(h-1)+x+1.$
			
			\item[(iv)]	If $\{x,y\}=\{r,k-h+1\}$ where $r \in [3,h-1]$, then
			\vspace{.1cm}
			\begin{center}
				$A=[0,r-1]\cup [r+1,k-h]\cup [k-h+2,k+1].$
			\end{center} 	\vspace{.1cm}
			Note that $hk- \dfrac{h(h-3)}{2}-1 \notin h^{\wedge}A$ and
			\vspace{.1cm}
			\begin{center}
				$h^{\wedge}A=\left[\dfrac{h(h+1)}{2}-r, hk - \dfrac{h(h-3)}{2}-2\right] \cup \left\{hk- \dfrac{h(h-3)}{2}\right\}.$
			\end{center}
			\noindent Therefore, 
			$|h^{\wedge}A|=hk-h^2+h+r.$
			
			\item [\upshape(v)]
			If $x \in [3,h-1]$ and $y \in [k-h+2,k-2]$, then \begin{center}
				$A=[0,x-1]\cup [x+1,y-1] \cup [y+1,k+1].$
			\end{center} It is easy to see that 
			\begin{center}
				$h^{\wedge}A=\bigg[\dfrac{h(h+1)}{2}-x, (h+1)k +1 - \dfrac{h(h-1)}{2}-y\bigg].$
			\end{center} Therefore, we have $|h^{\wedge}A|=(h+1)k-h^2+x-y+2.$
			
			\item[(vi)] 
			\begin{enumerate}
				\item If $\{x,y\}=\{h,s\}$ where $s \in [h+3,k-h]$, then \[A=[0,h-1] \cup [h+1,s-1]\cup [s+1,k+1].\] Note that $\dfrac{h(h-1)}{2} +1 \notin h^{\wedge}A$ and \begin{center}
					$h^{\wedge}A= \left\{\dfrac{h(h-1)}{2}\right\}   \bigcup \left[\dfrac{h(h-1)}{2}+2, hk-\dfrac{h(h-3)}{2}  \right]$.
				\end{center}
				Therefore, $|h^{\wedge}A|=hk-h^2+2h.$
				\item If $\{x,y\}=\{r,k-h+1\}$ where $r \in [h+1,k-h-2]$, then 
				\begin{align*}
					A &=[0,r-1] \cup [r+1,k-h] \cup [k-h+2,k+1]\\&
					= k+1 - ( [0,h-1] \cup [h+1,k-r]\cup [k-r+2,k+1])\\&
					= k+1 - ( [0,h-1] \cup [h+1,s-1]\cup [s+1,k+1])
				\end{align*}
				where $s=k+1-r \in [h+3,k-h]$. Since the cardinality of  $h^{\wedge}A$ is translation and dilation invariant, therefore, in both of the cases, we have 
				\[|h^{\wedge}A|=hk-h^2+2h.\]
			\end{enumerate}
			
			\item[(vii)] If $\{x,y\}=\{h,k-h+1\}$, then $A=[0,h-1] \cup [h+1,k-h] \cup [k-(h-2),k+1]$. Note that $\dfrac{h(h-1)}{2}+1$ and $hk-\dfrac{h(h-3)}{2}-1$ are not in $h^{\wedge}A$. Also, \begin{align*}
				h^{\wedge}A& =\left\{\dfrac{h(h-1)}{2}\right\} \bigcup \left[ \dfrac{h(h-1)}{2}+2, hk-\dfrac{h(h-3)}{2}-2 \right] \bigcup \left\{hk-\dfrac{h(h-3)}{2}\right\}
			\end{align*}
			Therefore, $|h^{\wedge}A|=hk-h^2+2h-1.$
			
			\item[(viii)] $\{x,y\}=\{h, s \}$ where $s \in [k-h+2,k-2]$, then 
			\begin{align*}
				A &=[0,h-1]\cup [h+1,s-1] \cup [s+1,k+1] \\&
				= (k+1) - ([0,k-s]\cup [k-s+2,k-h] \cup  [k-h+2,k+1])\\&
				= (k+1) - ([0,r-1]\cup [r+1,k-h]\cup [k-h+2,k+1]),
			\end{align*} 
			where $r=k+1-s \in [3,h-1] $. Since the cardinality of $h^{\wedge}A$ is translation and dilation invariant, we have \[|h^{\wedge}A|= |h^{\wedge}([0,r-1]\cup [r+1,k-h]\cup [k-h+2,k+1]),\] 
			where $r \in [3,h-1]$. Therefore,
			from the previous case, we have   \begin{align*}
				|h^{\wedge}A| &=hk-h^{2}+h+r\\ &=hk-h^2+h+k+1-s \\& 
				=(h+1)k-h(h-1)-s+1.\end{align*}
			
			\item [\upshape(ix)] If $x \in [h+1,k-(h+3)]$ and $y \in [h+4,k-h]$, then  \begin{center}
				$h^{\wedge}A=\bigg[\dfrac{h(h-1)}{2}, hk - \dfrac{h(h-3)}{2}\bigg].$
			\end{center} Therefore, we have
				$|h^{\wedge}A|=hk-h^2+2h+1.$

			\item[(x)] 
			If $x \in [h+1,k-h-1]$ and $y \in [k-h+2,k-2]$, then 
			\begin{align*}
				A&=[0,x-1] \cup [x+1,y-1] \cup [y+1,k+1]\\&
				=(k+1)-([0,k-y] \cup [k-y+2,k-x] \cup [k-x+2,k+1])\\&
				=(k+1)-([0,x'-1] \cup [x'+1,y'-1] \cup [y'+1,k+1]),
			\end{align*}
			where $x^{\prime} = k+1-y \in [3,h-1]$ and $y^{\prime} = k+1-x \in [h+2,k-h]$. Therefore, from \upshape(iii), we have 
			\begin{align*}
				|h^{\wedge}A|&= hk-h^2+h+x^{\prime}+1\\ & =  hk-h^2+h+k+1-y+1 \\&
				= (h+1)k -h(h-1)-y+2.
			\end{align*}

			\item [(xi)] If $\{x,y\}=\{k-h,s\}$ where $s \in [k-(h-3),k-2]$, then 
			\begin{align*}
				A &=[0,k-(h+1)]\cup [k-(h-1),s-1] \cup [s+1,k+1]
				\\ &
				= (k+1) - ([0,k-s]\cup [k-s+2,h] \cup [h+2,k+1])\\
				&= (k+1) - ([0,r-1]\cup [r+1,h] \cup [h+2,k+1]),
			\end{align*} 
			where $r = k+1-s \in [3,h-2] $. Since the cardinality of $h^{\wedge}A$ is translation and dilation invariant, we have \[|h^{\wedge}A|= |h^{\wedge}([0,r-1]\cup [r+1,h] \cup [h+2,k+1]),\] 
			where $r \in [3,h-2]$. Therefore, from \upshape(ii), we have 
			\begin{align*}
				|h^{\wedge}A|&=hk-h^2+h+r\\
				& = hk-h^{2}+h+k+1-s\\
				&=(h+1)k-h^{2}+h-s+1.
			\end{align*}
			
			\item [(xii)]  If $ x\in [k-h+1, k-5]$ and $y \in [k-h+4,k-2]$, then we have 
			\begin{align*}
				A &=[0,x-1]\cup [x+1,y-1] \cup [y+1,k+1] \\
				&=(k+1) - ([0,k-y] \cup [k-y+2, k-x] \cup [k-x+2,k+1])\\
				&=(k+1) - ([0,x^{\prime}-1] \cup [x^{\prime}+1,y^{\prime}-1] \cup [y^{\prime}+1,k+1]),
			\end{align*}
			where $x^{\prime} = k+1-y \in [3,h-3]$ and $y^{\prime} = k+1-x \in [6,h]$. Since the cardinality of $h^{\wedge}A$ is translation and dilation invariant, we have 
			\begin{center} $|h^{\wedge}A| = |h^{\wedge}([0,x'-1] \cup [x'+1,y'-1] \cup  [y'+1,k+1]),$ 
			\end{center}
			where $x' \in [3,h-3]$ and $y' \in [6,h]$. Therefore, from \upshape(i), we get that  $$|h^{\wedge}A|=hk-h^2+x^{\prime}+y^{\prime}=hk-h^2+2k+2-x-y.$$ 
		\end{enumerate}
	\end{proof}
	
	%
	%
		%
		%
		%
	
	\noindent In the following proposition, we compute the cardinality of $h^{\wedge}A$ when $A=[0,k+2] \setminus \{x,x+1,x+2\}, ~\text{where}~  x\in [1,k-1].$
	
	\begin{prop}\label{prop3.10}
		Let $ h \geq 3$ and $k \geq 3h+4 $ be positive integers. Let  \[A=[0,k+2] \setminus \{x,x+1,x+2\}, ~\text{where}~  x\in [1,k-1].\]
		\begin{enumerate}
			\item [\upshape(i)] If $x \in [1, h-1]$, then  $|h^{\wedge}A|= hk-h^2+3x+1$.
			
			\item [\upshape (ii)] If $x \in [k-h+1,k-1]$, then  $|h^{\wedge}A|=
			(h+3)k-h^2-3x+1$. 
			\item [\upshape(iii)]    If $x \in \{h,k-h\}$, then $|h^{\wedge}A|= hk-h^2+3h-2.$
			\item [\upshape(iv)]    If $x \in \{h+1,k-h-1\}$, then $|h^{\wedge}A|=
			\begin{cases}
				3k, & \text{ if } h=3;\\
				hk-h^2+3h+1, & \text{ if } h \geq 4.
			\end{cases}$
			\item [\upshape(v)]  If $x \in [h+2, k-h-2]$, then $|h^{\wedge}A|=hk-h^2+3h+1.$
			
		\end{enumerate}    
	\end{prop}
	
	\begin{proof}
		\begin{enumerate}
			\item [\upshape(i)]
			If $x \in [1, h-1]$, then $A=[0,x-1]\cup [x+3,k+2]$. It follows that
		 \begin{center}
				$h^{\wedge}A =
				\bigg[\dfrac{(h+2)(h+3)}{2}-(3x+3), hk - \dfrac{h(h-5)}{2}\bigg].$
			\end{center}
			Therefore,  
				$|h^{\wedge}A|=	hk-h^2+3x+1 .$

			\item[\upshape(ii)]  If $ x \in [k-h+1,k-1]$, then
			\begin{align*}
				A& =[0,x-1] \cup [x+3,k+2]\\ 
				&= (k+2)- ([0,k-x-1] \cup [k-x+3,k+2])\\ 
				&= (k+2)- ([0,y-1] \cup [y+3,k+2]), 
			\end{align*} where $y = k-x \in [1,h-1].$ Since the cardinality of $h^{\wedge}A$ is translation and dilation invariant, we have 
			\begin{center} $|h^{\wedge}A| = |h^{\wedge}([0,y-1] \cup [y+3,k+2])|, ~ \text{where}~ y \in [1,h-1].$
			\end{center} 
			Therefore, from previous case, we have  
			\begin{align*}
				|h^{\wedge}A| &=  hk-h^2+3y+1\\
				&= hk-h^2+3(k-x)+1\\
				& = (h+3)k-h^{2}-3x+1.
			\end{align*}

			\item 	[\upshape(iii)]  \begin{enumerate}
				\item 
				If $x=h$,  then $A=[0,h-1] \cup [h+3,k+2]$. It is easy to observe that
				$\dfrac{h(h-1)}{2}+1$, $\dfrac{h(h-1)}{2}+2$, and  $\dfrac{h(h-1)}{2}+3$ do not belongs to $h^{\wedge}A.$
				Also,
				\begin{center}
					$h^{\wedge}A = \left\{\dfrac{h(h-1)}{2}\right\} \bigcup\bigg[ \dfrac{h(h- 1)}{2}+4, hk-\dfrac{h(h-5)}{2}\bigg]. $
				\end{center}
				\item 	If $x=k-h$, then 
				\begin{align*} A&=[0,k-h-1] \cup [k-(h-3),k+2] \\
					& = (k+2)-([0,h-1] \cup [h+3,k+2]).
				\end{align*} Since the cardinality of $h^{\wedge}A$ is translation and dilation invariant, we have 
				\begin{center}
					$|h^{\wedge}A|= |h^{\wedge}([0,h-1] \cup [h+3,k+2])|.$ 
				\end{center} Therefore, in both the cases, we get that
					$|h^{\wedge}A| = hk-h^2+3h-2.$
			\end{enumerate}
			
			\item 	[\upshape(iv)] 
			\begin{enumerate}
				\item If $x=h+1$,  then $A=[0,h] \cup [h+4,k+2]$. Note that, for $h=3$, 
				\begin{align*}\dfrac{h(h- 1)}{2}+4 &=7 \notin 3^{\wedge}A.\end{align*}
				Therefore, $3^{\wedge}A=[3,6] \cup [8,3k+3].$  Let $h \geq 4$. It is evident that  
				\begin{align*}
					h^{\wedge}A &=\left[ \dfrac{h(h- 1)}{2}, hk-\dfrac{h(h-5)}{2}\right].
				\end{align*}
				Therefore, $|h^{\wedge}A|=
				\begin{cases}
					3k, & \text{ if } h=3;\\
					hk-h^2+3h+1, & \text{ if } h \geq 4.
				\end{cases}$
				\item 	If $x=k-h-1$, then we have \begin{align*}
					A &=[0,k-h-2]\cup [k-h+2, k+2] \\&
					= (k+2)-([0,h]\cup [h+4,k+2]).
				\end{align*} Since cardinality of $h^{\wedge}A$ is translation and dilation invariant, we have 
				\begin{align*}|h^{\wedge}A| &= |h^{\wedge}([0,h] \cup [h+4,k+2])|.
				\end{align*} 
				Therefore, from the previous case, we get that  \[|h^{\wedge}A|=
				\begin{cases}
					3k, & \text{ if } h=3;\\
					hk-h^2+3h+1, & \text{ if } h \geq 4.
				\end{cases}\] 
			\end{enumerate}
			
			\item 	[\upshape(v)]   If $x \in [h+2,k-h-2]$, then $A=[0,x-1]\cup [x+3,k+2]$. It is easy to get that \begin{align*}
				h^{\wedge}A &=\left[ \dfrac{h(h-1)}{2}, hk-\dfrac{h(h-5)}{2}\right].
			\end{align*} Therefore, $ |h^{\wedge}A| = hk-h^2+3h+1.$
		\end{enumerate}
	\end{proof}
	
\noindent	In the following proposition, we compute the cardinality of $3^{\wedge}A$ when $A=[0,k+2] \setminus \{x,y,z\},$ where  $2 \leq x<y<z\leq k$ with $y=x+1$ and $z-x \geq 3$.
	
	\begin{prop}\label{prop3.10-1}
		Let $k \geq 13$ be a positive integer. Let $$A=[0,k+2] \setminus \{x,x+1,z\} ~ \text{where} ~ x \in [2,k-3]~ \text{and} ~ z \in [5,k]~ \text{with} ~ z-x \geq 3.$$
		\begin{enumerate}
			\item [\upshape (i)] If $x=2$ and $z \in \{5,6,7,k-1,k\}$, then $|3^{\wedge}A|= 3k-2$.
			
			\item [\upshape (ii)] If $x=2$ and $z \in [8,k-2]$, then $\left|3^{\wedge}A\right| = 3k-1.$
			
			\item[\upshape (iii)] If $x=3$ and $z \in \{k-1,k\}$, then $|3^{\wedge}A|=3k-2$.
			
			\item[\upshape (iv)] If $x=3$ and $z \in [6,k-2]$, then $\left|3^{\wedge}A\right| = 3k.$
			\item[\upshape (v)] If $x \in [4,k-4]$ and $z = k-1$, then $\left|3^{\wedge}A\right| = 3k.$
			\item[\upshape (vi)] If $x \in [4,k-5]$ and $z = k$, then 
			$|3^{\wedge}A|=3k$. 
			\item[\upshape (vii)] If $x \in \{k-4,k-3\}$ and $z = k$, then
			$|3^{\wedge}A|=
			\begin{cases}
				3k-1, & \text{ if } x=k-4;\\
				3k-2, & \text{ if } x=k-3.
			\end{cases}$
			\item[\upshape (viii)] If $x \in [4,k-5]$ and $z \in [7,k-2]$, then $\left|3^{\wedge}A\right| = 3k+1.$
		\end{enumerate}
	\end{prop}		
	
	\begin{proof}
		\begin{enumerate}
			\item [(i)]
			\begin{enumerate}

				\item[\upshape(a)] If $x=2$ and $z \in \{5,6,7\} $, then 
				$A=\{0,1\}\cup [4,z-1] \cup [z+1,k+2]$. So, $3^{\wedge}A=[5,z] \cup [z+2,3k+3].$
				\item[\upshape(b)] If $x=2$ and $z =k-1$, then $A=\{0,1\}\cup [4,k-2] \cup \{k,k+1,k+2\}$. So, $3^{\wedge}A=[5,3k+1]\cup \{3k+3\}.$
				\item[\upshape(c)] If $x=2$ and $z =k$, then $A=\{0,1\}\cup [4,k-1] \cup \{k+1,k+2\}$. So, $3^{\wedge}A=[5,3k+2].$
				
				\noindent	Therefore, in each of the cases, we have $|3^{\wedge}A|=3k-2.$ 
			\end{enumerate}
			\item[\upshape(ii)] If $x=2$ and $z \in [8,k-2]$, then  
			$A=\{0,1\}\cup [4,z-1] \cup [z+1,k+2]$. It follows that  $3^{\wedge}A=[5,3k+3]$ and  $|3^{\wedge}A|=3k-1.$
			\item [(iii)]
			\begin{enumerate}

				\item[\upshape(a)] If $x=3$ and $z =k-1$, then $A=\{0,1,2\} \cup [5,k-2]\cup \{k,k+1,k+2\}$. So, $3^{\wedge}A=\{3\} \cup [6,3k+1] \cup \{3k+3\}.$
				\item[\upshape(b)] If $x=3$ and $z =k$, then $A=\{0,1,2\} \cup [5,k-1]\cup \{k+1,k+2\}$. So, $3^{\wedge}A=\{3\} \cup [6,3k+2].$
			\end{enumerate}
			\noindent	Therefore, in both of the cases, we have $|3^{\wedge}A|=3k-2$.
			
			\item [\upshape (iv)] If $x=3$ and $z \in [6,k-2]$, then $A= \{0,1,2\} \cup [4,z-1] \cup [z+1,k+2]$. So, $3^{\wedge}A=  \{3\} \cup [5,3k+3].$ Therefore $|3^{\wedge}A|=3k.$

			\item[\upshape(v)] If $x \in [4,k-4]$ and $z=k-1$, then $A=[0,x-1] \cup [x+2,k-2] \cup \{k,k+1,k+2\}$. It follows that $3^{\wedge}A=[3,3k+1] \cup \{3k+3\}$ and  $|3^{\wedge}A|=3k.$

			\item [\upshape(vi)] If $x \in [4,k-5]$ and $z=k$, then $A=[0,x-1] \cup [x+2,k-1] \cup \{k+1,k+2\}$. It follows that $3^{\wedge}A=[3,3k+2]$ and  $|3^{\wedge}A|=3k$.
			\item[\upshape(vii)]
			\begin{enumerate}
				\item [\upshape (a) ]If $x = k-4 $ and $z=k$, then $A=[0,k-5] \cup \{k-2,k-1,k+1,k+2\}$. It follows that  $3^{\wedge}A=[3,3k-1] \cup \{3k+1,3k+2\}$ and 	$|3^{\wedge}A|=	3k-1$. 
				
				\item [\upshape (b) ]If $x = k-3 $ and $z=k$, then $A=[0,k-4] \cup \{k-1,k+1,k+2\}$. It follows that  $3^{\wedge}A=[3,3k-1] \cup \{3k+2\}$ and 	$|3^{\wedge}A|=	3k-2$.
			\end{enumerate}
			\item[\upshape(viii)] If  $x \in [4,k-5]$ and $z \in [7,k-2]$, then $A=[0,x-1] \cup [x+2,z-1] \cup [z+1,k+2]$. It follows that $3^{\wedge}A=[3,3k+3]$ and  $|3^{\wedge}A|=3k+1$.
		\end{enumerate}
		This completes the proof of the proposition.
	\end{proof}	
	In the following proposition, we compute the cardinality of $3^{\wedge}A$ when $A=[0,k+2] \setminus \{x,y,z\},$ where  $2 \leq x<y<z\leq k$ with $z=y+1$ and $y-x \geq 2$.	
	Since $|h^{\wedge}A|$ is translation and dilation invariant, the proof of  Proposition \ref{prop3.10-2} follows from the Proposition \ref{prop3.10-1}
	\begin{prop}\label{prop3.10-2}
		Let $k \geq 13$ be a positive integer. Let $$A=[0,k+2] \setminus \{x,y,y+1\}, ~ \text{where} ~ x \in [2,k-3]~ \text{and} ~ y \in [4,k-1]~ \text{with} ~ y-x \geq 2.$$
		\begin{enumerate}
			\item [\upshape (i)]  If $x\in \{2,3,k-5,k-4,k-3\}$ and $y=k-1$, then 	$\left|3^{\wedge}A\right|= 3k-2.$
			\item [\upshape (ii)]  If $x\in [4,k-6]$ and $y=k-1$, then 	$\left|3^{\wedge}A\right| =3k-1$.
			
			\item[\upshape (iii)] If $x \in \{2,3\}$ and $y=k-2$, then $\left|3^{\wedge}A\right| = 3k-2$.

			\item[\upshape (iv)] If $x \in [4,k-4]$ and $y=k-2$, then $\left|3^{\wedge}A\right| = 
			3k-1$.
			\item [\upshape (v)] If $x=3$ and $y \in [5,k-3]$, then $|3^{\wedge}A|= 3k.$
			\item [\upshape (vi)] If $x=2$ and $y \in [6,k-3]$, then $|3^{\wedge}A|=3k$.

			\item [\upshape (vii)] If $x=2$ and $y \in \{4,5\}$, then 	$|3^{\wedge}A|=
			\begin{cases}
				3k-1, & \text{ if } y=5;\\
				3k-2, & \text{ if } y=4.
			\end{cases}$
			\item[\upshape (viii)] If $x \in [4,k-5]$ and $y \in [6,k-3]$, then $\left|3^{\wedge}A\right| = 3k+1$.
		\end{enumerate}
	\end{prop}				
	
\noindent	In the Proposition \ref{3 elmnt gen case} and Proposition \ref{prop3.12}, we compute $|3^{\wedge}A|$ when $A=[0,k+2] \setminus \{x,y,z\}, ~ \text{where} ~ 2 \leq x<y<z\leq k  ~ \text{with} ~ y-x \geq 2  ~ \text{and} ~ z-y \geq 2.$
	\begin{prop}\label{3 elmnt gen case}
		Let $k \geq 13$ be a positive integer. Let $$A=[0,k+2] \setminus \{x,y,z\}, ~ \text{where} ~ 2 \leq x<y<z\leq k  ~ \text{with} ~ y-x \geq 2  ~ \text{and} ~ z-y \geq 2.$$
		\begin{enumerate}
			\item [\upshape (i)] If $\{x,y,z\}$ is one the sets  $\{2,4,6\}$, $\{2,4,k-1\}$, $\{2,4,k\}$, $\{2,5,k-1\}$, and $\{2,5,k\}$, then $|3^{\wedge}A|=3k-2.$
			
			\item [\upshape (ii)] If $\{x,y,z\}$ is one the sets $\{2,4,t\}$ where $t \in [7,k-2]$, $\{2,5,t\}$ where $t \in [7,k-2] $, $\{2,s,k-1\}$ where $s \in [6,k-3]$, $\{2,s,k\}$ where $s \in [6,k-4]$, and $\{3,s,k-1\}$ where $s \in [5,k-3]$, then $|3^{\wedge}A|=3k-1.$
			
			\item [\upshape (iii)] If $\{x,y,z\}$ is one the sets $\{r,s,k\}$ where $5 \leq r+1<s\leq k-4$, and $\{3,s,t\}$ where $6 \leq s+1< t \leq k-2$, then $|3^{\wedge}A|=3k.$
			
			\item [\upshape (iv)] If $\{x,y,z\}= \{r,s,t\}$ where $4\leq r<s<t\leq k-2 ~ \text{and} ~ s-r \geq 2, ~ t-s \geq 2$, then $|3^{\wedge}A|=3k+1.$
		\end{enumerate}
	\end{prop}
	\begin{proof}
		\begin{enumerate}
			\item [\upshape(i)] 
			\begin{enumerate}
				\item If $\{x,y,z\}=\{2,4,6\}$, then $A=\{0,1,3,5\}\cup [7,k+2]$ and $3^{\wedge}A=\{4,6\} \cup [8,3k+3].$
				\item If $\{x,y,z\}=\{2,4,k-1\}$, then $A=\{0,1,3\}\cup [5,k-2] \cup [k,k+2]$ and $3^{\wedge}A=\{4\} \cup [6,3k+1]\cup \{3k+3\}.$
				\item If $\{x,y,z\}=\{2,4,k\}$, then $A=\{0,1,3\}\cup [5,k-1] \cup \{k+1,k+2\}$ and $3^{\wedge}A=\{4\} \cup [6,3k+2].$
				\item If $\{x,y,z\}=\{2,5,k-1\}$, then $A=\{0,1,3,4\}\cup [6,k-2] \cup [k,k+2]$ and $3^{\wedge}A=\{4,5\} \cup [7,3k+1]\cup \{3k+3\}.$
				\item If $\{x,y,z\}=\{2,5,k\}$, then $A=\{0,1,3,4\}\cup [6,k-1] \cup \{k+1,k+2\}$ and $3^{\wedge}A=\{4,5\} \cup [7,3k+2].$
			\end{enumerate}
			Therefore, in each of the cases, we have $|3^{\wedge}A|=3k-2.$

			\item[\upshape(ii)]
			\begin{enumerate}
				\item If $\{x,y,z\}=\{2,4,t\}$, where $ t \in [7,k-2]$, then 
				$A=\{0,1,3\}\cup [5,t-1] \cup [t+1,k+2]$ and $3^{\wedge}A=\{4\} \cup [6,3k+3].$
				\item If $\{x,y,z\}=\{2,5,t\}$,  $ t \in [7,k-2]$, then 
				$A=\{0,1,3,4\}\cup [6,t-1] \cup [t+1,k+2]$ and $3^{\wedge}A=\{4,5\} \cup [7,3k+3].$
				\item If $\{x,y,z\}= \{2,s,k-1\}$, where $ s \in [6,k-3]$, then
				$A=\{0,1\}\cup [3,s-1] \cup [s+1,k-2]\cup [k,k+2]$
				and $3^{\wedge}A=[4,3k+1]\cup \{3k+3\}.$
				\item If $\{x,y,z\}=\{2,s,k\}$, where $ s \in [6,k-4]$, then 
				$A=\{0,1\}\cup [3,s-1] \cup [s+1,k-1]\cup \{k+1,k+2\}$ and $3^{\wedge}A=[4,3k+2].$
				\item If $\{x,y,z\}= \{3,s,k-1\}$, where $ s \in [5,k-3]$, then 
				$A=\{0,1,2\}\cup [4,s-1] \cup [s+1,k-2]\cup [k,k+2]$
				and $3^{\wedge}A=\{3\} \cup [5,3k+1]\cup \{3k+3\}.$
			\end{enumerate}
			Therefore, in each of the cases, we have $|3^{\wedge}A|=3k-1.$
			\item[\upshape (iii)] 
			\begin{enumerate}
				\item If $\{x,y,z\}=\{r,s,k\}$, where $5 \leq r+1<s \leq k-4$, then
				$A=[0,r-1] \cup [r+1,s-1] \cup [s+1,k-1] \cup \{k+1,k+2\}$ 
				and $3^{\wedge}A=[3,3k+2].$
				\item If $\{x,y,z\}=\{3,s,t\}$, where $6 \leq s+1<t \leq k-2$, then
				$A=\{0,1,2\} \cup [4,s-1] \cup [s+1,t-1] \cup [t+1,k+2]$ 
				and $3^{\wedge}A=\{3\} \cup [5,3k+3].$
			\end{enumerate}
			Therefore, in both the cases, we have $|3^{\wedge}A|=3k.$
			\item[\upshape (iv)] 
			
			If $\{x,y,z\}=\{r,s,t\}$, where $ 4\leq r<s<t\leq k-2 ~ \text{and} ~ s-r \geq 2,~ t-s \geq 2$, then 
			$A=[0,r-1] \cup [r+1,s-1] \cup [s+1,t-1] \cup [t+1,k+2]$ and $3^{\wedge}A=[3,3k+3].$ 
			Therefore,  $|3^{\wedge}A|=3k+1.$
		\end{enumerate}
	\end{proof}
	
	Since $|3^{\wedge}A|$ is translation invariant, the proof of the following proposition follows from the Proposition \ref{3 elmnt gen case}.
	\begin{prop}\label{prop3.12}
		Let $k \geq 11$ be a positive integer. Let \begin{align*}
			A=[0,k+2] \setminus \{x,y,z\}, ~ \text{where} ~ 2 \leq x<y<z\leq k. 
		\end{align*}
		\begin{enumerate}
			\item [\upshape (i)] If $\{x,y,z\}$ is one of the sets  $\{k-4,k-2,k\}$, $\{3,k-2,k\}$, $\{2,k-2,k\}$,  $\{2,k-3,k\}$, and $\{3,k-3,k\}$, then $|3^{\wedge}A|=3k-2.$
			
			\item [\upshape (ii)] If $\{x,y,z\}$  is one of the sets  $\{r,k-2,k\}$ with $r \in [4,k-5]$, $\{r,k-3,k\}$ with $r \in [4,k-5]$, and $\{3,s,k\}$ with $s \in [5,k-4]$, then $|3^{\wedge}A|=3k-1.$
			
			\item [\upshape (iii)] If $\{x,y,z\}$  is one of the sets $\{2,s,t\}$ with $7 \leq s+1<t\leq k-2$, and $\{r,s,k-1\}$ with $5 \leq r+1< s \leq k-3$, then $|3^{\wedge}A|=3k.$
		\end{enumerate}
		
	\end{prop}
	
\noindent	In the following proposition, we compute the cardinality of $3^{\wedge}A$ when $A=[0,k+2] \setminus \{1,y,z\}, ~ \text{where} ~ 2 \leq y<z\leq k+1$. Since $|3^{\wedge}A|$ is translation and dilation invariant, proof of Proposition \ref{prop3.15} follows from Proposition \ref{3 case for x=1}.

	\begin{prop}\label{3 case for x=1}
		Let $k \geq 13$ be a positive integer. Let $$A=[0,k+2] \setminus \{1,y,z\}, ~ \text{where} ~ 2 \leq y<z\leq k+1.$$
		\begin{enumerate}
			\item [ \upshape(i)] If $\{y,z\} =\{2,k+1\}$, then $|3^{\wedge}A|=3k-5$.
			\item [ \upshape(ii)] If $\{y,z\} $ is one of the sets $\{2,4\}$, $\{2,5\}$, $\{2,k-1\}$, $\{2,k\}$, $\{3,k+1\}$, and $\{4,k+1\}$, then $|3^{\wedge}A|=3k-4$.
			\item [ \upshape(iii)] If $\{y,z\} $ is one of the sets $\{3,4\}$, $\{3,5\}$, $\{3,6\}$, $\{3,k-1\}$, $\{3,k\}$, $\{4,5\}$, $\{4,k-1\}$, $\{4,k\}$, $\{k-1,k\}$, $\{k-2,k-1\}$, $\{k-3,k\}$, $\{k-2,k\}$, $\{2,t\}$ where $t \in [6, k-2]$, and $\{s,k+1\}$ where $s \in [5, k-3]$, then $|3^{\wedge}A|=3k-3$. 
			\item [\upshape (iv)] If $\{y,z\} $ is one of the sets $\{5,6\}$, $\{3,t\}$ where $t \in[7, k-2]$, $\{4,t\}$ where  $t \in [6,k-2]$, $\{s,k-1\}$ where $s \in [5,k-3]$, and $\{s,k\}$ where $s \in [5,k-4]$,  then $|3^{\wedge}A|=3k-2$.
			\item [\upshape (v)] If $\{y,z\}$ is one of the sets $\{s,s+1\}$ where $s \in[6,k-3]$, and $\{s,t\}$ where $6 \leq s+1<t \leq k-2$, then  $|3^{\wedge}A|=3k-1$.
		\end{enumerate}
		\begin{proof}
			\begin{enumerate}
				\item [ \upshape(i)] If $\{y,z\}=\{2,k+1\}$, then $A=\{0\} \cup [3,k] \cup \{k+2\}$ and $3^{\wedge}A=[7,3k+1].$
				Therefore,
				$|3^{\wedge}A|=3k-5.$		
				
				\item [ \upshape(ii)]
				\begin{enumerate}
					\item If $\{y,z\}=\{2,4\}$, then $A=\{0,3\} \cup [5,k+2]$ and $3^{\wedge}A=[8,3k+3].$
					\item If $\{y,z\}=\{2,5\}$, then $A=\{0,3,4\} \cup [6,k+2]$ and $3^{\wedge}A=\{7\} \cup [9,3k+3].$
					\item If $\{y,z\}=\{2,k-1\}$, then $A=\{0\} \cup [3,k-2] \cup [k,k+2]$ and $3^{\wedge}A=[7,3k+1]\cup \{3k+3\}.$
					\item If $\{y,z\}=\{2,k\}$, then $A=\{0\} \cup [3,k-1] \cup \{k+1,k+2\}$ and $3^{\wedge}A= [7,3k+2].$
					\item If $\{y,z\}=\{3,k+1\}$, then $A=\{0,2\} \cup [4,k] \cup \{k+2\}$ and $3^{\wedge}A= [6,3k+1].$
					\item If $\{y,z\}=\{4,k+1\}$, then $A=\{0,2,3\} \cup [5,k] \cup \{k+2\}$ and $3^{\wedge}A= \{5\} \cup [7,3k+1].$
				\end{enumerate}
				Therefore, in each of the cases, we have
			$|3^{\wedge}A|=3k-4.$
				
				\item [ \upshape(iii)]
				\begin{enumerate}
					\item If $\{y,z\}=\{3,4\}$, then $A=\{0,2\} \cup [5,k+2]$ and $3^{\wedge}A=[7,3k+3].$
					\item If $\{y,z\}=\{3,5\}$, then $A=\{0,2,4\} \cup [6,k+2]$ and $3^{\wedge}A=\{6\} \cup [8,3k+3].$		
					\item If $\{y,z\}=\{3,6\}$, then $A=\{0,2,4,5\} \cup [7,k+2]$ and $3^{\wedge}A=\{6,7\} \cup [9,3k+3].$
					\item If $\{y,z\}=\{3,k-1\}$, then $A=\{0,2\} \cup [4,k-2] \cup [k,k+2]$ and $3^{\wedge}A= [6,3k+1]\cup \{3k+3\}.$
					\item If $\{y,z\}=\{3,k\}$, then $A=\{0,2\} \cup [4,k-1] \cup \{k+1,k+2\}$ and $3^{\wedge}A= [6,3k+2].$
					\item If $\{y,z\}=\{4,5\}$, then $A=\{0,2,3\} \cup [6,k+2]$ and $3^{\wedge}A= \{5\} \cup [8,3k+3].$
					\item If $\{y,z\}=\{4,k-1\}$, then $A=\{0,2,3\} \cup [5,k-2] \cup [k,k+2]$ and $3^{\wedge}A= \{5\} \cup [7,3k+1] \cup \{3k+3\}.$
					\item If $\{y,z\}=\{4,k\}$, then $A=\{0,2,3\} \cup [5,k-1] \cup \{k+1,k+2\}$ and $3^{\wedge}A= \{5\} \cup [7,3k+2].$
					\item If $\{y,z\}=\{k-1,k\}$, then $A=\{0\} \cup [2,k-2] \cup \{k+1,k+2\}$, and $3^{\wedge}A= [5,3k+1].$
					\item If $\{y,z\}=\{k-2,k-1\}$, then $A=\{0\} \cup [2,k-3] \cup [k,k+2]$ and $3^{\wedge}A= [5,3k] \cup \{3k+3\}.$
					\item If $\{y,z\}=\{k-3,k\}$, then $A=\{0\} \cup [2,k-4] \cup \{k-2,k-1,k+1,k+2\}$ and $3^{\wedge}A= [5,3k-1] \cup \{3k+1,3k+2\}.$
					\item If $\{y,z\}=\{k-2,k\}$, then $A=\{0\} \cup [2,k-3] \cup \{k-1,k+1,k+2\}$ and $3^{\wedge}A= [5,3k] \cup \{3k+2\}.$
					\item If $\{y,z\}=\{2,t\}$, where $t \in [6,k-2]$, then $A=\{0\} \cup [3,t-1] \cup [t+1,k+2]$ and $3^{\wedge}A= [7,3k+3].$
					\item If $\{y,z\}=\{s,k+1\}$, where $s \in [5,k-3]$, then $A=\{0\} \cup [2,s-1] \cup [s+1,k] \cup \{k+2\}$ and $3^{\wedge}A= [5,3k+1].$
				\end{enumerate}
				Therefore, in each of the cases, we have
					$|3^{\wedge}A|=3k-3.$
				
				\item [ \upshape(iv)]
				\begin{enumerate}
					\item If $\{y,z\}=\{5,6\}$, then $A=\{0\} \cup [2,4] \cup [7,k+2]$ and $3^{\wedge}A=[5,7] \cup [9,3k+3].$
					\item If $\{y,z\}=\{3,t\}$, where $t \in [7,k-2]$, then $A=\{0,2\} \cup [4,t-1] \cup [t+1,k+2]$ and $3^{\wedge}A= [6,3k+3].$
					\item If $\{y,z\}=\{4,t\}$, where $t \in [6,k-2]$, then $A=\{0,2,3\} \cup [5,t-1] \cup [t+1,k+2]$ and $3^{\wedge}A= \{5\} \cup [7,3k+3].$
					\item If $\{y,z\}=\{s,k-1\}$, where $s \in [5,k-3]$, then $A=\{0\} \cup [2,s-1] \cup [s+1,k-2] \cup [k,k+2]$ and $3^{\wedge}A= [5,3k+1] \cup \{3k+3\}.$
					\item If $\{y,z\}=\{s,k\}$, where $s \in [5,k-4] $, then $A=\{0\} \cup [2,s-1] \cup [s+1,k-1] \cup \{k+1,k+2\}$ and $3^{\wedge}A= [5,3k+2].$
				\end{enumerate}
				Therefore, in each of the cases, we have
					$|3^{\wedge}A|=3k-2.$
				
				\item [\upshape (v)]
				\begin{enumerate}
					\item If $\{y,z\}=\{s,s+1\}$, where $s \in [6,k-3]$, then $A=\{0\} \cup [2,s-1] \cup [s+2,k+2]$ and $3^{\wedge}A=[5,3k+3]$.
					\item If $\{y,z\}=\{s,t\}$, where $6 \leq s+1<t \leq k-2$, then $A=\{0\} \cup [2,s-1] \cup [s+1,t-1] \cup [t+1,k+2]$ and $3^{\wedge}A=[5,3k+3]$.
				\end{enumerate}
				Therefore, in both the cases, we have
					$|3^{\wedge}A|=3k-1.$
			\end{enumerate}		
		\end{proof}
	\end{prop}
	
	\begin{prop}\label{prop3.15}
		Let $k \geq 11$ be a positive integer. Let $$A=[0,k+2] \setminus \{x,y,k+1\}, ~ \text{where} ~ 1 \leq x<y\leq k.$$
		\begin{enumerate}
			\item [ \upshape(i)] If $\{x,y\}=\{1,k\}$, then $|3^{\wedge}A|=3k-5$.
			
			\item [ \upshape(ii)] If $\{x,y\}$ is one of the sets $\{k-2,k\}$, $\{k-3,k\}$, $\{3,k\}$, $\{2,k\}$, $\{1,k-1\}$, and $\{1,k-2\}$, then $|3^{\wedge}A|=3k-4$.
			
			\item [ \upshape(iii)] If $\{x,y\} $ is one of the sets $\{k-2,k-1\}$, $\{k-3,k-1\}$, $\{k-4,k-1\}$, $\{3,k-1\}$, $\{2,k-1\}$, $\{k-3,k-2\}$, $\{3,k-2\}$, $\{2,k-2\}$, $\{2,3\}$, $\{3,4\}$, $\{2,5\}$, $\{2,4\}$, $\{r,k\}$ where $ r \in [4,k-4]$, and $\{1,s\}$ where $ s \in [5,k-3]$, then $|3^{\wedge}A|=3k-3$. 
			
			\item [\upshape (iv)] If $\{x,y\}$ is one of the sets $\{k-4,k-3\}$, $\{r,k-1\}$ where $r \in [4,k-5]$, $\{r,k-2\}$ where $r \in [4, k-4]$, $\{3,s\}$ where $s \in [5,k-3]$, and $\{2,s\}$ where $s \in [6,k-3],$  then $|3^{\wedge}A|=3k-2$.
			
			\item [\upshape (v)] If $\{x,y\}$ is one of the sets $\{r,r+1\}$ where $r \in [4,k-5]$, and $\{r,s\}$ where $5 \leq r+1<s \leq k-3$, then  $|3^{\wedge}A|=3k-1$.
		\end{enumerate}
	\end{prop}
	
\noindent	In the following propositions, we compute the cardinality of $h^{\wedge}A$ for $A=[0,k+2] \setminus \{x,x+1,k+1\}, ~ \text{where} ~ 1 \leq x \leq k-1$ and $h \geq 4$. Since $|h^{\wedge}A|$ is translation and dilation invariant, proof of  Proposition \ref{prop 3.16(i)} follows from Proposition \ref{prop 3.17(i)}.
	
	\begin{prop}\label{prop 3.16(i)}
		Let $h \geq 4$ and $k \geq 3h+4$ be positive integers. Let $$A=[0,k+2] \setminus \{x,x+1,k+1\}, ~ \text{where} ~ 1 \leq x \leq k-1.$$
		\begin{enumerate}
			\item [\upshape (i)] If $x \in [1,h-1]$, then $|h^{\wedge}A|= hk - h^2 + 2x +2.$
			\item [\upshape (ii)] If $x \in \{h,k-h\}$, then $|h^{\wedge}A|= hk - h^2 + 2h.$
			\item [\upshape (iii)] If $x \in [h+1,k-h-1]$, then $|h^{\wedge}A|= hk - h^2 + 2h +2.$
			\item [\upshape (iv)] If $x \in [k-(h-1),k-1]$, then $|h^{\wedge}A|= (h+2)k - h^2 - 2x+2.$
		\end{enumerate}
	\end{prop}
	\begin{proof}
		\begin{enumerate}
			\item [\upshape (i)] If $x \in [1,h-1]$, then $A= [0,x-1] \cup [x+2, k] \cup \{k+2\}$. It is easy to see that $$h^{\wedge}A=\bigg[\dfrac{h(h+3)}{2}-2x, hk - \dfrac{h(h-3)}{2}+1\bigg].$$
			Therefore, we have $|h^{\wedge}A|= hk - h^2 + 2x +2.$
			\item [\upshape (ii)] 
			\begin{enumerate}
				\item [\upshape (a) ] If $x=h$, then $A=[0,h-1] \cup [h+2,k] \cup \{k+2\}$. Observe that, $\dfrac{h(h-1)}{2} +1$ and $\dfrac{h(h-1)}{2} +2$ do not belong to $h^{\wedge}A$. Clearly, $$h^{\wedge}A= \bigg\{\dfrac{h(h-1)}{2}\bigg\} \bigcup \bigg[\dfrac{h(h-1)}{2} +3, hk - \dfrac{h(h-3)}{2}+1 \bigg].$$
				\item [\upshape (b) ] If $x=k-h$, then $A=[0,k-(h+1)]\cup [k-(h-2),k] \cup \{k+2\}.$ Observe that, $hk - \dfrac{h(h-3)}{2}-1$, $hk- \dfrac{h(h-3)}{2}$ do not belong to $h^{\wedge}A$. Clearly, $$h^{\wedge}A= \bigg[\dfrac{h(h-1)}{2}, hk - \dfrac{h(h-3)}{2}-2 \bigg] \bigcup \bigg\{hk - \dfrac{h(h-3)}{2}+1\bigg\}.$$
			\end{enumerate}
			Therefore, in both the cases $|h^{\wedge}A|= hk-h^2+2h.$
			\item [\upshape (iii)] If $x \in [h+1,k-h-1]$, then $A=[0,x-1] \cup [x+2,k] \cup \{k+2\}$. It is easy to see that $$h^{\wedge}A=\bigg[\dfrac{h(h-1)}{2}, hk - \dfrac{h(h-3)}{2}+1\bigg].$$
			Therefore, $|h^{\wedge}A|= hk-h^2+2h+2.$
			
			\item[\upshape (iv)] If $x \in [k-(h-1),k-1],$ then $A=[0,x-1] \cup [x+2,k] \cup \{k+2\}$. Clearly, $$h^{\wedge}A= \bigg[\dfrac{h(h-1)}{2}, (h+2)k - \dfrac{h(h+1)}{2}-2x+1\bigg]. $$
			Therefore, $|h^{\wedge}A|=(h+2)k -h^2-2x+2.$
		\end{enumerate}
	\end{proof}
	
	\begin{prop}\label{prop 3.17(i)}
		Let $h \geq 4$ and $k \geq 3h+2$ be positive integers. Let $$A=[0,k+2] \setminus \{1,y,y+1\}, ~ \text{where} ~ 2 \leq y \leq k.$$
		\begin{enumerate}
			\item [\upshape(i) ] If $y \in [k-(h-2),k]$, then $|h^{\wedge}A|= hk-h^2+2(k+1-y).$
			\item [\upshape (ii)] If $y \in \{h+1, k-h+1\}$, then $|h^{\wedge}A|= hk-h^2+2h.$
			\item [\upshape (iii)] If $y \in [h+2,k-h]$, then $|h^{\wedge}A|= hk-h^2+2h+2.$
			\item [\upshape (iv)] If $y \in [2,h]$, then $|h^{\wedge}A|= hk-h^2+2y.$
		\end{enumerate}
	\end{prop}
	
\noindent	In the following propositions, we compute the cardinality of $h^{\wedge}A$ for $A=[0,k+2] \setminus \{x,y,z\}, ~ \text{where} ~ 2 \leq x<y<z\leq k$ and $h \geq 4$.
	
	\begin{prop}\label{prop 3.13}
		Let $h \geq 4$ and $k \geq 3h+4$ be positive integers. Let $$A=[0,k+2] \setminus \{x,x+1,z\}, ~ \text{where} ~ 2 \leq x<z\leq k ~ \text{with} ~ z-x \geq 3.$$
		\begin{enumerate}
			\item [\upshape(i)] If $ x \in [2,h-2]$ and $z \in [5,h+1]$, then $|h^{\wedge}A|= hk-h^2+2x+z-1.$
			\item [\upshape (ii)] If $ x \in [2,h-1]$ and $z \in [h+2,k-h+1]$, then 
			$$|h^{\wedge}A| =
			\begin{cases}
				hk-h(h-1)+2x, ~ & ~ \text{if} ~ x \in [2,h-1] ~ \text{and} ~ z = h+2;\\
				hk-h(h-1)+2x+1, ~ & ~ \text{if}~ x \in [2,h-2] ~ \text{and} ~ z \in [h+3,k-h+1];\\
				hk-h(h-3)-2,~ & ~ \text{if}~ x=h-1 ~ \text{and} ~ z \in \{h+3,h+4\};\\
				hk-h(h-3)-1,~ & ~ \text{if}~ x=h-1 ~ \text{and} ~ z \in [h+5,k-h+1].
			\end{cases}$$
			\item [\upshape(iii)] If $x=h$ and $z \in [h+3,k-h+1]$, then $|h^{\wedge}A| = hk-h(h-3)-1.$
			
			\item [\upshape (iv)] If $ x \in [2,h]$ and $z = k-h+2$, then 
			\[\left|h^{\wedge}A\right| = 
			\begin{cases}
				hk-h(h-1)+2x, & ~ \text{if} ~ x \in [2,h-1]; \\
				hk-h(h-3)-2,   & ~ \text{if} ~ x=h.
			\end{cases}
			\]
			\item [\upshape (v)] If $ x \in [2,h]$ and $z \in [k-h+3,k]$, then
			\[\left|h^{\wedge}A\right| = 
			\begin{cases}
				(h+1)k -h^2 + 2x -z +3, & ~ \text{if} ~ x \in [2,h-1]; \\
				(h+1)k -h(h-2) -z +1,   & ~ \text{if} ~ x=h.
			\end{cases}
			\]
			
			\item [\upshape(vi)] If $ x \in [h+1,k-h-1]$ and $z \in [h+4,k-h+2]$, then 
			\[\left|h^{\wedge}A\right| = 
			\begin{cases}
				hk-h(h-3)+1, & ~ \text{if} ~ z \in [h+4,k-h+1]; \\
				hk-h(h-3),   & ~ \text{if} ~ z = k-h+2.
			\end{cases}
			\]
			
			\item [\upshape(vii)] If $ x \in [h+1,k-h]$ and $z \in [k-h+3,k]$, then
			\[\left|h^{\wedge}A\right| = 
			\begin{cases}
				(h+1)k-h(h-2)-z+3, & ~ \text{if} ~ x \in [h+1,k-h-1]; \\
				(h+1)k-h(h-2)-z+1,   & ~ \text{if} ~ x = k-h.
			\end{cases}
			\]

			
			\item [\upshape (viii)] If $ x \in [k-h+1,k-3]$ and $z \in [k-h+4,k]$, then $|h^{\wedge}A|= (h+3)k-h^2-(2x+z)+3.$
		\end{enumerate}
	\end{prop}
	
	\begin{proof}
		\begin{enumerate}
			\item [\upshape(i).]
			
			If $ x \in [2,h-2]$ and $z \in [5,h+1]$, then $A=[0,x-1]\cup [x+2,z-1] \cup [z+1,k+2]$. It is easy to see that \begin{center}
				$h^{\wedge}A =
				\bigg[\dfrac{h(h+5)}{2}-(2x+z)+2, hk - \dfrac{h(h-5)}{2}\bigg].$
			\end{center} 
			Therefore, we have 
				$|h^{\wedge}A|=	hk-h^2+(2x+z)-1.$

			\item[\upshape (ii)] 
			\begin{enumerate}
				\item [\upshape (a)] If $x\in [2,h-1]$ and $z=h+2$, then $A = [0,x-1] \cup [x+2,h+1] \cup [h+3, k+2]$. Note that $\dfrac{h(h+3)}{2}-2x+1\notin h^{\wedge}A$ and 	$$	h^{\wedge}A =\left\{\dfrac{h(h+3)}{2}-2x\right\} \bigcup \left[\dfrac{h(h+3)}{2}-2x+2,hk - \dfrac{h(h-5)}{2} \right].$$ Therefore,
					$|h^{\wedge}A|=	hk-h^2+h+2x.$
				\item[\upshape (b)]If $x\in [2,h-2]$ and $z=[h+3,k-h+1]$, then $A = [0,x-1] \cup [x+2,z-1] \cup [z+1, k+2]$. Therefore	$$	h^{\wedge}A = \left[\dfrac{h(h+3)}{2}-2x,hk - \dfrac{h(h-5)}{2} \right].$$ Therefore,
					$|h^{\wedge}A|=	hk-h^2+h+2x+1.$
				\item [\upshape(c)]
					If $x = h-1$ and $z=h+3$, then $A = [0,h-2] \cup [h+1,h+2] \cup [h+4,k+2]$. Note that $\dfrac{h(h-1)}{2}+4 \notin h^{\wedge}A$ and \begin{center}$h^{\wedge}A = \left\{\dfrac{h(h-1)}{2}+2, \dfrac{h(h-1)}{2}+3\right\} \bigcup \left[\dfrac{h(h-1)}{2}+5, hk - \dfrac{h(h-5)}{2}\right].$
					\end{center} 
					 If $x = h-1$ and $z=h+4$, then $A = [0,h-2] \cup [h+1,h+3] \cup [h+5,k+2]$. Note that $\dfrac{h(h-1)}{2}+5 \notin h^{\wedge}A$ and 
					\begin{center}
						$h^{\wedge}A =  \left[\dfrac{h(h-1)}{2}+2, hk - \dfrac{h(h-5)}{2}\right] \setminus \left\{\dfrac{h(h-1)}{2}+5\right\} .$
					\end{center}
				Therefore, in both the cases, we have $|h^{\wedge}A|=hk-h(h-3)-2.$
				\item[\upshape (d)]	If $ x =h-1$ and $z \in [h+5, k-h+1]$. Then $A = [0,h-2] \cup [h+1, z-1] \cup [z+1,k+2]$. It is easy to see that \begin{center}
					$h^{\wedge}A =
					\bigg[\dfrac{h(h-1)}{2}+2, hk - \dfrac{h(h-5)}{2}\bigg].$
				\end{center} 
				Therefore, we have  $|h^{\wedge}A|=hk-h(h-3)-1.$
			\end{enumerate}
			\item [\upshape (iii)] If $x=h$ and $z \in [h+3, k-h+1]$, then $A= [0,h-1] \cup [h+2,z-1] \cup [z+1,k+2]$. Observe that, $ \dfrac{h(h-1)}{2}+1$ and $\dfrac {h(h-1)}{2}+2$ do not belong to $h^{\wedge}A .$ Clearly,  \begin{center}
				$h^{\wedge}A = \left\{\dfrac{h(h+1)}{2}\right\} \bigcup 
				\bigg[\dfrac{h(h-1)}{2}+3, hk - \dfrac{h(h-5)}{2}\bigg].$
			\end{center} 
			Therefore, we have 
				$|h^{\wedge}A|=	hk-h(h-3)-1.$	
			\item[ \upshape (iv)]
			\begin{enumerate}
				\item[\upshape (a)]
				If $ x \in [2,h-1]$ and $z =k-h+2$, then
				\begin{align*}
					A= [0,x-1] \cup [x+2,k-h+1] \cup [k-h+3,k+2].
				\end{align*} 
				Observe that, $hk - \dfrac{h(h-5)}{2}-1 \notin h^{\wedge}A$.
				Clearly, 
				\begin{align*}
					h^{\wedge}A & =  
					\bigg[\dfrac{(h+1)(h+2)}{2}-(2x+1), hk - \dfrac{h(h-5)}{2}-2\bigg] \bigcup \left\{hk - \dfrac{h(h-5)}{2}\right\}.
				\end{align*} 
				Therefore, we have 
					$|h^{\wedge}A|=	hk-h(h-1)+2x.$
				\item[\upshape(b)] If $x=h$ and $z= k-h+2$, then \begin{align*}
					A=[0,h-1] \cup [h+2,k-h+1] \cup [k-h+3,k+2].
				\end{align*} Observe that $\dfrac{h(h-1)}{2}+1$, $\dfrac {h(h-1)}{2}+2$, and $hk - \dfrac{h(h-5)}{2}-1$ do not belong to $h^{\wedge}A.$ Also, 
				\begin{align*}
					h^{\wedge}A = \left\{\dfrac{h(h-1)}{2}\right\} \bigcup 
					\bigg[\dfrac{h(h-1)}{2}+3, hk - \dfrac{h(h-5)}{2}-2\bigg] \bigcup \left\{hk - \dfrac{h(h-5)}{2}\right\}.
				\end{align*} 
				Therefore, 
					$|h^{\wedge}A|=	hk-h(h-3)-2.$
			\end{enumerate}
			
			\item[\upshape (v)] 
			\begin{enumerate}
				\item [\upshape (a)] If $x \in [2,h-1]$ and $z \in [k-h+3,k]$, then $A=[0,x-1] \cup [x+2,z-1] \cup [z+1,k+2]$. Clearly, 
				\begin{center}
					$h^{\wedge}A = \left[ \dfrac{(h+1)(h+2)}{2} -(2x+1), (h+1)k- \dfrac{(h-2)(h-1)}{2}-z+3\right].$
				\end{center}
				Therefore, 
					$|h^{\wedge}A|=	(h+1)k-h^2+2x-z+3.$
				
				\item [\upshape (b)] If $x=h$ and $z \in [k-h+3,k]$, then $A=[0,h-1] \cup [h+2,z-1] \cup [z+1,k+2]$. Note that $ \dfrac{h(h-1)}{2}+1$ and $\dfrac {h(h-1)}{2}+2$ are not in $h^{\wedge}A$ and \begin{center}
					$h^{\wedge}A = \left\{\dfrac{h(h-1)}{2}\right\} \bigcup 
					\bigg[\dfrac{h(h-1)}{2}+3, (h+1)k - \dfrac{(h-2)(h-1)}{2}-z+3\bigg].$
				\end{center}
				Therefore,
					$|h^{\wedge}A|=	(h+1)k-h(h-2)-z+1.$
			\end{enumerate}
			
			\item[\upshape(vi)] 
			\begin{enumerate}
				\item[\upshape(a)]
				If $x \in [h+1,k-h-1]$ and $z \in [h+4,k-h+1]$, then clearly, 
				\begin{center}
					$h^{\wedge}A = \left[ \dfrac{h(h-1)}{2}, hk - \dfrac{h(h-5)}{2}\right].$ 
				\end{center}
				Therefore, $|h^{\wedge}A|=	hk-h^2+3h+1.$
				
				\item [\upshape(b)] If $x \in [h+1,k-h-1]$ and $z=k-h+2$, then $A=[0,x-1] \cup [x+1,k-h+1] \cup [k-h+3,k+2]$. Observe that $hk - \dfrac{h(h-5)}{2}-1 \notin h^{\wedge}A$. Evidently, 
				\begin{center}
					$h^{\wedge}A = \left[ \dfrac{h(h-1)}{2}, hk - \dfrac{h(h-5)}{2}-2\right] \bigcup \left\{hk - \dfrac{h(h-5)}{2}\right\}.$ 
				\end{center}
				Therefore, we have $|h^{\wedge}A|=	hk-h^2+3h.$
			
			\end{enumerate}
			\item[\upshape (vii)] 
			\begin{enumerate}
				\item[\upshape(a)]
				If $x \in [h+1,k-h-1]$ and $z \in [k-h+3,k]$, then $A=[0,x-1] \cup [x+2,z-1] \cup [z+1,k+2]$. Clearly, 
				\begin{center}
					$h^{\wedge}A =
					\bigg[\dfrac{h(h-1)}{2}, (h+1)k - \dfrac{(h-2)(h-1)}{2}-z+3\bigg].$
				\end{center} 
				Therefore, we have 
					$|h^{\wedge}A|=	(h+1)k-h(h-2)-z+3.$
				\item [\upshape(b)]
				If $x =k-h$ and $z \in [k-h+3,k]$, then $A=[0,k-h-1] \cup [k-h+2,z-1] \cup [z+1,k+2]$. Observe that, 
				$(h+1)k - \dfrac{h(h-3)}{2}-z$ and $(h+1)k- \dfrac{h(h-3)}{2}-z+1$ do not belong in $h^{\wedge}A$, and
				\begin{center}
					$h^{\wedge}A =
					\bigg[\dfrac{h(h-1)}{2}, (h+1)k - \dfrac{h(h-3)}{2}-z-1\bigg] \bigcup \left\{(h+1)k - \dfrac{h(h-3)}{2}-z+2\right\}.$
				\end{center}
				Therefore,
					$|h^{\wedge}A|=	(h+1)k-h(h-2)-z+1.$
			\end{enumerate}
			
			\item[\upshape (viii)] If $ x \in [k-h+1, k-3]$ and $ z \in [k-h+4,k] $, then  \begin{center}
				$h^{\wedge}A =
				\left[\dfrac{h(h-1)}{2}, (h+3)k - \dfrac{h(h+1)}{2} -(2x+z) +2\right].$
			\end{center}
			Therefore,
			$|h^{\wedge}A| = (h+3)k-h^2-(2x+z)+3.$
		
		\end{enumerate}
	\end{proof}
	
	\begin{prop}
		Let $h \geq 4$ and $k \geq 3h+4$ be two integers. Let $$A=[0,k+2] \setminus \{x,y,y+1\}, ~ \text{where} ~ 2 \leq x<y\leq k-1 ~ \text{with} ~ y-x \geq 2.$$
		\begin{enumerate}
			\item [\upshape(i)] If $x \in [k-h+1,k-3]$ and $y \in [k-h+3,k-1]$, then $|h^{\wedge}A|= (h+3)k-h^2-(x+2y)+3.$ 
			\item [\upshape (ii)] If $x \in [h+1,k-h]$ and $y \in [k-h+2,k-1]$, then $$|h^{\wedge}A|=
			\begin{cases}
				(h+2)k-h(h-1)-2y+2, ~ & ~ \text{if} ~ x=k-h, y=k-h+2,~ \text{for} ~ h \geq 5 ~ \text{and} \\
				 ~ & ~ x \in \{k-h-1,k-h\},y=k-h+2, ~ \text{for}~ h=4;\\
				(h+2)k-h(h-1)-2y+3, ~ & ~ \text{otherwise}.
			\end{cases} $$ 
			\item[\upshape(iii)] If $x \in [h+1,k-h-1]$ and $y=k-h+1$, then $|h^{\wedge}A|= hk-h(h-3)-1.$
			\item[\upshape(iv)] If $x=h$ and $y \in [k-h+1,k-1]$, then \[\left|h^{\wedge}A\right| = 
			\begin{cases}
				(h+2)k-h(h-1)-2y+2, & ~ \text{if} ~ y \in [k-h+2,k-1]; \\
				hk-h(h-3)-2,   & ~ \text{if} ~ y=k-h+1.
			\end{cases}
			\]
			
			\item[\upshape (v)] If $ x \in [2,h-1]$ and $y \in [k-h+1,k-1]$, then
			\[\left|h^{\wedge}A\right| = 
			\begin{cases}
				(h+2)k -h^2 + x -2y +3, & ~ \text{if} ~ y \in [k-h+2,k-1]; \\
				hk -h(h-2) +x -1,   & ~ \text{if} ~ y=k-h+1.
			\end{cases}
			\]
			
			\item[\upshape (vi)] If $ x \in [h,k-h-2]$ and $y \in [h+2,k-h]$, then 
			\[\left|h^{\wedge}A\right| = 
			\begin{cases}
				hk-h^2+3h+1, & ~ \text{if} ~ x \in [h+1,k-h-2]; \\
				hk-h^2+3h,   & ~ \text{if} ~ x = h.
			\end{cases}
			\]
			\item[\upshape(vii)] If $ x \in [2,h-1]$ and $y \in [h+1,k-h]$, then
			\[\left|h^{\wedge}A\right| = 
			\begin{cases}
				hk-h(h-2)+x+1, & ~ \text{if} ~ y \in [h+2,k-h]; \\
				hk-h(h-2)+x-1,   & ~ \text{if} ~ y = h+1.
			\end{cases}
			\]
			\item[\upshape(viii)] If $ x \in [2,h-2]$ and $y \in [4,h]$, then $|h^{\wedge}A|= hk-h^2+x+2y-1.$
		\end{enumerate}
		
	\end{prop}
	
	Since $|h^{\wedge}A|$ is translation invariant, the proof follows from Proposition \ref{prop 3.13}.

	\begin{prop}
		Let $h \geq 4$ and $k \geq 3h+4$ be positive integers. Let $$A=[0,k+2] \setminus \{x,y,z\}, ~ \text{where} ~ x,y,z\in [2,k] ~ \text{with} ~ y-x \geq 2 ~ \text{and} ~ z-y \geq 2.$$
		\begin{enumerate}
			\item[ \upshape (i)] If $2 \leq x<y<z \leq h+1$, then $|h^{\wedge}A|= hk-h^2+(x+y+z)-2.$ 
			\item [ \upshape (ii)] If $k-h+1 \leq x<y<z \leq k$, then $|h^{\wedge}A|=(h+3)k -h^2-(x+y+z)+4.$
			
			\item[ \upshape (iii)] If $x \in [2,h-1]$, $y \in [4,h+1]$, and $ z \in [k-(h-3),k]$, then \[|h^{\wedge}A|= 
			\begin{cases}
				(h+1)k-h^2+(x+y-z)+2, & ~ \text{if} ~ x \in [2,h-2], ~ y \in [4,h+1]; \\
				(h+1)k -h(h-2) -z +1, & ~ \text{if} ~ x=h-1, ~ y=h+1.
			\end{cases}
			\]
			\item [\upshape (iv) ]  If $x \in [2,h-2]$, $y \in [4,h]$ and $ z \in [h+2,k-(h-2)]$, then 
			\[\left|h^{\wedge}A\right| = 
			\begin{cases}
				hk-h(h-1)+x+y, & ~ \text{if} ~ z \in [h+2,k-(h-1)]; \\
				hk-h(h-1)+x+y-1,   & ~ \text{if} ~ z = k-(h-2).
			\end{cases}
			\]

			\item[ \upshape (v)] If  $x \in [2,h-1]$, $y \in [k-(h-2),k-2]$  and $z \in [k-(h-4), k]$, then $|h^{\wedge}A|=(h+2)k-h^2 +(x-y-z) +4. $
			
			\item[ \upshape (vi)] If $x \in [2,h-1]$, $y\in [h+1, k-(h-1)]$, $ z \in [k-(h-3),k]$, then \[|h^{\wedge}A|= 
			\begin{cases}
				(h+1)k-h(h-1)+(x-z)+3, & ~ \text{if} ~ y \in [h+2,k-h];\\
				
				(h+1)k-h(h-1)+(x-z)+2, ~ & ~ \text{if} ~ y= h+1 ~ \text{or} ~ k-(h-1). 
			\end{cases}	\] 
			
			\item [ \upshape (vii)] If $x \in [2,h-1]$, $y\in [h+1,k-h]$, and $z=k-(h-2)$, then 
			\[|h^{\wedge}A|=
			\begin{cases}
				hk-h(h-2)+x-1, ~ & ~ \text{if} ~ y=h+1;\\
				hk-h(h-2)+x, ~ & ~ \text{if} ~ y \in [h+2,k-h].
			\end{cases} \]
			
			\item[ \upshape (viii)] If $x \in [2,h-1]$, $y \in [h+1,k-h]$, and $ z \in [h+3,k-(h-1)]$, then \[|h^{\wedge}A|=	
			\begin{cases}
				hk-h(h-3)-2, & ~  \text{if} ~ \{x,y,z\}=\{h-1, h+1,h+3\};\\
				hk-h(h-3)+1,  & ~  \text{if} ~ \{x,y,z\}=\{h-1, h+1, k-(h-1)\};\\
				hk-h(h-2)+x, & ~  \text{if} ~ x \in [2,h-1], ~ y=h+1, ~ z \in [h+4,k-(h-1)];\\
				hk-h(h-2)+x+1, & ~ \text{if} ~ x \in [2,h-1], ~ y \in [h+2,k-h-2], ~ z \in [h+4,k-(h-1)].
			\end{cases}\]
			
			\item [\upshape (ix)] If $x=h$, $y \in [h+2,k-h]$ and $z \in [h+4,k]$, then \[|h^{\wedge}A|=
			\begin{cases}
				(h+1)k-h(h-2)-z+2, ~ & ~ \text{if} ~ z \in [k-(h-3),k];\\
				hk-h(h-3)-1, ~ & ~ \text{if} ~ z=k-(h-2);\\
				hk-h(h-3), ~ & ~ \text{if} ~ z \in [h+4,k-(h-1)]
			\end{cases}
			\]
			\item [\upshape (x)] If $x=h$, $y \in [k-(h-1),k-2]$, and $z \in [k-(h-3),k]$, then \[|h^{\wedge}A|= 
			\begin{cases}
				hk-h(h-3)-2, ~ & ~ \text{if} ~ \{x,y,z\}=\{h,k-(h-1),k-(h-3)\};\\
				(h+1)k-h(h-2)-z+1, ~ & ~ \text{if} ~ x=h, ~ y=k-(h-1), ~ z \in [k-(h-4),k];\\
				hk-h(h-3)-3, ~ & ~ \text{if} ~ \{x,y,z\}=\{h,k-(h-2),k-(h-4)\};\\
				(h+1)k-h(h-2)-z+1, ~ & ~ \text{if} ~ x=h, ~ y=k-(h-2), ~ z \in [k-(h-5),k];\\
				(h+2)k -h(h-1)-(y+z) +3, ~ & ~ \text{if} ~ x=h, ~ y\in [k-(h-3),k-2], ~ z \in [k-(h-5),k].
			\end{cases}\]


			\item[\upshape(xi)] If $h+1 \leq x<y \leq k-(h-1)$ and $ z \in [k-(h-3),k]$, then \[|h^{\wedge}A|=
			\begin{cases}
				hk-h(h-3)-2, ~ & ~ \text{if} ~ \{x,y,z\}= \{k-(h+1), k-(h-1), k-(h-3)\};\\
				hk-h(h-3)-1, ~ & ~ \text{if} ~ x \in [h+1,k-(h+2)],~ y=k-(h-1),~ z=k-(h-3);\\
				(h+1)k -h(h-2) -z+3, ~ & ~ \text{otherwise}.
			\end{cases}
			\]
			\item [\upshape (xii)] If $x \in [h+1,k-h]$ and $y,z \in [k-(h-2),k]$, then \[|h^{\wedge}A|=
			\begin{cases}
				(h+2)k-h(h-1)-(y+z)+3, ~ & ~ \text{if} ~ x=k-h;\\
				(h+2)k-h(h-1)-(y+z)+4, ~ & ~ \text{otherwise}.
			\end{cases}
			\]

			\item[ \upshape (xiii)] If $x,y,z \in [h+1,k-(h-2)]$, then  \[|h^{\wedge}A|=
			\begin{cases}
				hk-h(h-3),  ~ & ~ \text{if} ~ z=k-(h-2);\\
				hk-h(h-3)+1, ~ & ~ \text{otherwise}. 
			\end{cases}
			\]

		\end{enumerate}
	\end{prop}
	\begin{proof}
		\begin{enumerate}
			\item [\upshape(i).]
			
			If $x,y,z \in [2, h+1]$, then $A=[0,x-1]\cup [x+1,y-1] \cup [y+1,z-1] \cup [z+1,k+2]$. It is easy to see that \begin{center}
				$h^{\wedge}A =
				\bigg[\dfrac{(h+2)(h+3)}{2}-(x+y+z), hk - \dfrac{h(h-5)}{2}\bigg].$
			\end{center} 
			Therefore, we have	$|h^{\wedge}A|=	hk-h^2+(x+y+z)-2.$
		
			\item [\upshape(ii).] If $ x,y,z \in [k-h+1,k]$, then we have  \begin{align*}
				A &=[0,x-1]\cup [x+1,y-1] \cup [y+1,z-1] \cup [z+1,k+2] \\&
				=(k+2)- ([0,x'-1] \cup [x'+1,y'-1] \cup [y'+1,z'-1]\cup [z'+1,k+2]), 
			\end{align*} where $x',y',z' \in [2,h+2].$ Since cardinality of $h^{\wedge}A$ is translation invariant, we have 
			\begin{center}$|h^{\wedge}A| = |h^{\wedge}([0,x'-1] \cup [x'+1,y'-1] \cup [y'+1,z'-1]\cup [z'+1,k+2]),$
			\end{center}
			where $x',y',z' \in [2,h+2].$ Therefore, in this case, we have 
			\begin{center}
				$|h^{\wedge}A| =hk-h^2-2+\{3k+6-(x+y+z)\} =(h+3)k-h^2-(x+y+z)+4.$
			\end{center}

			\item 	[\upshape(iii).] 
			\begin{enumerate}
				\item [\upshape (a)]
				If $x \in [2,h-2]$, $y \in [4,h]$, and $z \in [k-(h-3),k]$,  then \begin{center}
					$A=[0,x-1] \cup [x+1,y-1] \cup [y+1,z-1] \cup [z+1,k+2].$
				\end{center}  So, it is easy to see that  \begin{center}$h^{\wedge}A =\bigg[ \dfrac{(h+1)(h+2)}{2}-(x+y), (h+1)k +3-\dfrac{(h-1)(h-2)}{2}-z\bigg].$
				\end{center}
				Therefore, we have
					$|h^{\wedge}A|= (h+1)k-h^2+(x+y-z)+2$.
			\item [\upshape (b)]
				If $x=h-1$, $y=h+1$ and $z \in [k-(h-3),k]$,  then \begin{center}
					$A=[0,h-2] \cup \{h\} \cup [h+2,z-1] \cup [z+1,k+2].$
				\end{center}
				Observe that,  $\dfrac{h(h-1)}{2}+2$ does not belong to $h^{\wedge}A$. Clearly, 
				\begin{center}
					$h^{\wedge}A=\left\{\dfrac{h(h-1)}{2}+1\right\} \bigcup \left[\dfrac{h(h-1)}{2}+3, (h+1)k +3-\dfrac{(h-1)(h-2)}{2}-z  \right].$
				\end{center}
				Therefore, $|h^{\wedge}A|=(h+1)k-h(h-2)-z+1.$ 
			\end{enumerate}
			\item 	[\upshape(iv).]
			\begin{enumerate}
				\item [\upshape (a) ]
				If $x \in [2,h-2]$, $y \in [4,h]$ and $z \in [h+2, k-(h-1)]$, then 
				\begin{center}
					$A=[0,x-1] \cup [x+1,y-1] \cup [y+1,z-1] \cup [z+1,k+2].$
				\end{center}
				Clearly, 
				\begin{center}
		$h^{\wedge}A= \bigg[\dfrac{h(h+3)}{2}-(x+y)+1, hk- \dfrac{h(h-5)}{2}\bigg].$
			\end{center}
				Therefore, $|h^{\wedge}A|=hk-h(h-1)+x+y.$
				
				\item [\upshape(b) ] If $x \in [2,h-2]$, $y \in [4,h]$ and $z=k-(h-2)$, then 
				\begin{center}
					$A=[0,x-1] \cup [x+1,y-1] \cup [y+1,k-(h-1)] \cup [k-(h-3),k+2].$
				\end{center}
				Observe that, $hk- \dfrac{h(h-5)}{2}-1$ does not belong to $h^{\wedge}A$. Clearly, 
				\begin{center}
				$h^{\wedge}A= \bigg[\dfrac{h(h+3)}{2}-(x+y)+1, hk- \dfrac{h(h-5)}{2}-2\bigg] \bigcup \bigg\{ hk- \dfrac{h(h-5)}{2} \bigg\}. $
				\end{center}
				Therefore, $|h^{\wedge}A|=hk-h(h-1)+x+y-1.$	
				
			\end{enumerate}
			\item [\upshape(v).] If $x\in [2,h-1]$, $y \in [k-(h-2),k-2]$, and $y \in [k-(h-3),k]$, then we have 
			\begin{center}
				$	A =[0,x-1] \cup [x+1,y-1] \cup [y+1,z-1] \cup [z+1,k+2].$
			\end{center}
			Clearly, \begin{center}
				$	h^{\wedge}A	=\left[\dfrac{h(h+1)}{2}-x, (h+2)k - \dfrac{h(h-1)}{2} - (y+z)+3\right]$
			\end{center} 
			Therefore,
				$|h^{\wedge}A|= (h+2)k - h^2+(x-y-z)+4.$
			\item [\upshape(vi).]
			\begin{enumerate}
				\item [\upshape (a) ] If $x \in [2,h-1]$, $y\in [h+2, k-h]$, and $z \in [k-(h-3),k]$, then 
				\begin{center}
					$A=[0,x-1]\cup [x+1,y-1] \cup [y+1,z-1] \cup [z+1,k+2].$
				\end{center}
				Clearly, $$h^{\wedge}A= \left[ \dfrac{h(h+1)}{2}-x, (h+1)k +3 - \dfrac{(h-1)(h-2)}{2}-z\right].$$
				Therefore $|h^{\wedge}A|=(h+1)k-h(h-1)+(x-z)+3.$
				\item [\upshape (b)]
				\begin{enumerate}
					\item
					If $x \in [2,h-1]$, $y=h+1$ and $z \in [k-(h-3),k]$, then 
					\begin{center}
						$A=[0,x-1]\cup [x+1,h] \cup [h+2,z-1] \cup [z+1,k+2].$
				\end{center} Observe that $\dfrac{h(h+1)}{2}-x+1$ does not belong to $h^{\wedge}A$. It is easy to see that \begin{multline*} h^{\wedge}A= \left \{ \dfrac{h(h+1)}{2}-x \right\} \bigcup \bigg[\dfrac{h(h+1)}{2}-x+2, (h+1)k +3 - \dfrac{(h-1)(h-2)}{2}-z\bigg] .
					\end{multline*}

					\item	If $x \in [2,h-1]$, $y=k-(h-1)$ and $z \in [k-(h-3),k]$, then 
					\begin{center}
						$A=[0,x-1]\cup [x+1,k-h] \cup [k-h+2,z-1] \cup [z+1,k+2].$
					\end{center} Note that $(h+1)k -\dfrac{(h-1)(h-2)}{2}-z+2$ does not belong to $h^{\wedge}A$. Clearly, \begin{multline*} 
						h^{\wedge}A=  \left[\dfrac{h(h+1)}{2}-x, (h+1)k - \dfrac{(h-1)(h-2)}{2}-z+1\right]\\ \bigcup  \left\{(h+1)k - \dfrac{(h-1)(h-2)}{2}-z+3\right\} .
					\end{multline*} 
				\end{enumerate}
				
				Therefore, in both the cases, we have \begin{align*}
					|h^{\wedge}A|&=hk-h^2+k+2+(x+h-z)\\&=(h+1)k-h(h-1)+(x-z)+2.
				\end{align*}
			\end{enumerate}
			
			\item [\upshape(vii).] We take $z=k-(h-2)$, then regardless of the choices of $x$ and $y$, $hk- \dfrac{(h-2)(h-3)}{2}+2$ does not belong to $ h^{\wedge}A$.
			\begin{enumerate}
				\item [\upshape (a)]
				If $x \in [2,h-1]$ and $y=h+1$, then
				$	A =[0,x-1]\cup [x+1,h] \cup [h,k-(h-1)] \cup [k-(h-3),k+2].$	Note that 
				$\dfrac{h(h+1)}{2}-x+1$ is not contained in $ h^{\wedge}A$ as well.
				Evidently, \begin{multline*}
					h^{\wedge}A =  \left\{\dfrac{h(h+1)}{2}-x\right\}  \bigcup \left[\dfrac{h(h+1)}{2}-x+2, hk- \dfrac{(h-2)(h-3)}{2}+1\right] \\ \hspace{2.5cm}\bigcup  \left\{hk- \dfrac{(h-2)(h-3)}{2}+3\right\}.
				\end{multline*}
				Therefore, we have   $|h^{\wedge}A|=hk-h^2+2h+x-1.$
				
				\item [\upshape (b)] If $x \in [2,h-1]$ and $y \in [h+2, k-h]$, then 
				$	A =[0,x-1]\cup [x+1,y-1] \cup [y+1,k-(h-1)] \cup [k-(h-3),k+2]$ and 
				 \begin{center}$h^{\wedge}A= \left[\dfrac{h(h+1)}{2}-x,hk- \dfrac{(h-2)(h-3)}{2}+1 \right] \bigcup \left\{hk- \dfrac{(h-2)(h-3)}{2}+3\right\}.$
				\end{center} Therefore, $|h^{\wedge}A|=hk-h(h-2)+x. $
			\end{enumerate}
			\item[\upshape (viii). ]	In the first three cases, when $y=h+1$, no matter which $x$ is chosen, $\dfrac{h(h+1)}{2}-x+1$ does not belong to $h^{\wedge}A$.
			\begin{enumerate}
				\item [\upshape (a) ] If $x=h-1$ and $z=h+3$, then $A=[0,h-2] \cup \{h,h+2\} \cup [h+4, k+2]$. Observe that $\dfrac{h(h-1)}{2}+4$ does not belong to $h^{\wedge}A$. Clearly, 
				\begin{center}
					$h^{\wedge}A=\left\{\dfrac{h(h-1)}{2}+1,\dfrac{h(h-1)}{2}+3\right\} \bigcup \left[\dfrac{h(h-1)}{2}+5, hk- \dfrac{h(h-5)}{2} \right].$
				\end{center}
				Therefore, $|h^{\wedge}A|= hk-h(h-3)-2.$
				\item [\upshape (b) ] If $x=h-1$ and $z=k-(h-1)$, then $A=[0,h-2] \cup \{h\} \cup [h+2, k-h] \cup [k-(h-2),k+2]$ and 
				\begin{center}
					$h^{\wedge}A= \left\{\dfrac{h(h-1)}{2}+1\right\} \bigcup \left [ \dfrac{h(h-1)}{2}+3, hk- \dfrac{h(h-5)}{2} \right].$
				\end{center}
				Therefore, $|h^{\wedge}A|= hk-h(h-3)+1. $
				
				\item [\upshape (c) ] If $x \in [2,h-1]$ and $z \in [h+4,k-(h-1)]$, then $A= [0,x-1] \cup [x+1,h] \cup [h+2,z-1] \cup [z+1,k+2]$ and 
				\begin{center}
					$h^{\wedge}A= \left\{\dfrac{h(h+1)}{2}-x\right\} \bigcup \left[\dfrac{h(h+1)}{2}-x+2, hk- \dfrac{h(h-5)}{2} \right]$.
				\end{center}
				Therefore, $|h^{\wedge}A|= hk-h(h-2)+x.$			
				\item [\upshape (d)]  If $x \in [2,h-1]$, $y\in [h+2,k-h-1]$, and $z \in [h+4,k-(h-1)]$, then $A= [0,x-1] \cup [x+1,y-1] \cup [y+1,z-1] \cup [z+1,k+2]$. Clearly, \begin{center}
					$h^{\wedge}A=\left[\dfrac{h(h+1)}{2}-x, hk- \dfrac{h(h-5)}{2} \right].$\end{center}
				Therefore, $|h^{\wedge}A|= hk-h(h-2)+x+1.$	
				
			\end{enumerate}
			\item[\upshape (ix). ]	Here, we take $x=h$ and regardless of the choices of $y$ and $z$, $ \dfrac{h(h+1)}{2}+1$ is not an element of  $h^{\wedge}A$. 
			\begin{enumerate}
				\item [\upshape (a)] If $y \in [h+2,k-(h+1)]$, and $z \in [k-(h-3),k]$, then $A=[0,h-1] \cup [h+1,y-1] \cup [y+1,z-1] \cup [z+1,k+2]$ and 
				\begin{center}
					$h^{\wedge}A=\left\{\dfrac{h(h+1)}{2}\right\} \bigcup \left[\dfrac{h(h+1)}{2}+2, (h+1)k+3- \dfrac{(h-1)(h-2)}{2}-z\right].$
				\end{center}  Therefore, $|h^{\wedge}A|= (h+1)k - h(h-2)-z+2.$
				\item [\upshape (b)] If $y \in [h+2,k-h]$, and $z=k-(h-2)$, then $A=[0,h-1] \cup [h+1,y-1] \cup [y+1,k-(h-1)] \cup [k-(h-3),k+2]$. Note that, $(h+1)k+2- \dfrac{(h-1)(h-2)}{2}-z$ does not belong to $h^{\wedge}A$. Clearly, 
				\begin{multline*}
					h^{\wedge}A = \left\{ \dfrac{h(h+1)}{2}\right\} \bigcup \left[ \dfrac{h(h+1)}{2}+2,(h+1)k+1- \dfrac{(h-1)(h-2)}{2}-z\right] \\ \hspace{2cm} \bigcup \left\{(h+1)k+3- \dfrac{(h-1)(h-2)}{2}-z\right\}.
				\end{multline*}
				Therefore, $|h^{\wedge}A|= hk-h(h-3)-	1. $
				\item  [\upshape (c)]  If $y \in [h+2,k-h]$, and $z=[h-4,k-(h-1)]$, then $A=[0,h-1] \cup [h+1,y-1] \cup [y+1,z-1] \cup [z+1,k+2]$ and \begin{center}$h^{\wedge}A=\left\{\dfrac{h(h+1)}{2}\right\} \bigcup \left[\dfrac{h(h+1)}{2}+2, hk- \dfrac{h(h-5)}{2}\right].$
				\end{center} Therefore, $|h^{\wedge}A|= hk - h(h-3).$
			\end{enumerate}		
			
			\item[\upshape(x). ] As we have mentioned in the earlier case, for $x=h$, $ \dfrac{h(h+1)}{2}+1 \notin h^{\wedge}A$.
			\begin{enumerate}
				\item[\upshape(a)] If $y=k-(h-1)$ and $z=k-(h-3)$, then $A=[0,h-1] \cup [h+1,k-h] \cup \{k-(h-2)\} \cup [k-(h-4),k+2]$. Observe that, $hk-\dfrac{h(h-5)}{2} -2$ does not belong to $h^{\wedge}A$. Evidently,
				\begin{multline*}
					h^{\wedge}A =\left\{\dfrac{h(h+1)}{2}\right\} \bigcup \left[ \dfrac{h(h+1)}{2}+2,hk-\dfrac{h(h-5)}{2} -3 \right]\bigcup \left\{hk- \dfrac{h(h-5)}{2}-1\right\}.
				\end{multline*}
				Therefore, $ |h^{\wedge}A|= hk - h(h-3) -2.$
				\item[\upshape(b)] If $y=k-(h-1)$ and $z \in [k-(h-4),k]$. Note that, $\dfrac{hk-h(h-5)}{2} -2$ does not belong to $h^{\wedge}A$. Clearly, 
				\begin{multline*}
					h^{\wedge}A =\left\{\dfrac{h(h+1)}{2}\right\} \bigcup \left[ \dfrac{h(h+1)}{2}+2,(h+1)k-\dfrac{(h-1)(h-2)}{2} -z+1 \right]\\ \bigcup \left\{(h+1)k- \dfrac{(h-1)(h-2)}{2} -z+3\right\},
				\end{multline*}
				Therefore, $|h^{\wedge}A|= (h+1)k-h(h-2)-z+1.$
				\item[\upshape(c)] If $y=k-(h-2)$ and $z = k-(h-4)$, then $A=[0,h-1] \cup [h+1,k-(h-1)] \cup \{k-(h-3)\} \cup [k-(h-5),k]$. Clearly,
				\begin{center}
					$h^{\wedge}A= \left\{\dfrac{h(h-1)}{2}\right\} \bigcup \left[\dfrac{h(h-1)}{2}+2, hk - \dfrac{h(h-5)}{2}-3\right].$
				\end{center}
				Therefore, $|h^{\wedge}A|=hk-h(h-3)-3.$
				\item[\upshape (d)] If $y=k-(h-2)$ and $z \in [k-(h-5),k]$, then $A=[0,h-1] \cup [h+1,k-(h-1)] \cup [k-(h-3),z-1] \cup [z+1,k+2]$ and  \begin{center}
					$h^{\wedge}A= \left\{\dfrac{h(h-1)}{2}\right\} \bigcup \left[\dfrac{h(h-1)}{2}+2, (h+1)k - \dfrac{h(h-3)}{2}-z+1\right],$
				\end{center}
				Therefore, $|h^{\wedge}A|=(h+1)k-h(h-2)-z+1.$

				\item[\upshape (e)]	If $y \in [k-(h-3),k-2]$ and $z \in [k-(h-5),k]$, then $A=[0,h-1] \cup [h+1,y-1] \cup [y+1,z-1] \cup [z+1,k+2]$. Clearly,
				\begin{center}
					$ h^{\wedge}A= \left\{\dfrac{h(h-1)}{2}\right\} \bigcup \left[\dfrac{h(h-1)}{2} +2, (h+2)k - \dfrac{h(h-1)}{2} - (y+z) +3  \right].$
				\end{center}
				Therefore,  $ |h^{\wedge}A|= (h+2)k - h(h-1) -(y+z) +3.$
				
			\end{enumerate}

			\item [\upshape(xi).] 
			\begin{enumerate}
				\item [\upshape (a)] If $x=k-(h+1)$, $y=k-(h-1)$, and $z=k-(h-3)$, then $A=[0,k-h-2] \cup \{k-h,k-(h-2)\} \cup  [k-(h-4),k+2]$. Note that $hk- \dfrac{h(h-5)}{2} -2$ and $hk- \dfrac{h(h-5)}{2} - 4$ do not belong to $h^{\wedge}A$. Clearly, 
				\begin{multline*}
					h^{\wedge}A= \left[ \dfrac{h(h-1)}{2}, hk- \dfrac{h(h-5)}{2} -5\right] \bigcup \left\{hk- \dfrac{h(h-5)}{2} -3, hk- \dfrac{h(h-5)}{2} -1\right\}.
				\end{multline*}
				Therefore, $|h^{\wedge}A|= hk - h(h-3) -2.$
				\item [\upshape (b)] If $x \in [h+1, k-(h+2)]$, $y=k-(h-1)$, and $z=k-(h-3)$, then $A=[0,x-1] \cup [x+1,k-h] \cup \{k-(h-2)\} \cup [k-(h-4),k+2]$. Note that $hk- \dfrac{h(h-5)}{2} -2$ does not belong to $h^{\wedge}A$. Clearly, \begin{center}
					$h^{\wedge}A= \left[ \dfrac{h(h-1)}{2}, hk- \dfrac{h(h-5)}{2} -3\right] \bigcup \left\{hk- \dfrac{h(h-5)}{2} -1\right\}$.
				\end{center}
				Therefore, $|h^{\wedge}A|= hk - h(h-3) -1.$
				\item[\upshape (c)]
				If $x \in [h+1,k-(h+1)]$, $y \in [h+3,k-(h-1)]$, and $z \in [k-(h-3),k]$, but $y=k-(h-1)$ and $z=k-(h-3)$ do not occur at the same time. Then \begin{align*}
					A &=[0,x-1]\cup [x+1,y-1] \cup [y+1,z-1] \cup [z+1,k+2] \\&
					= (k+2)- [0,x'-1]\cup [x'+1,y'-1] \cup [y'+1,z'-1] \cup [z'+1,k+2],
				\end{align*} where $x' \in [2,h-1]$ and $y',z' \in [h+2,k-h].$  Since cardinality of $h^{\wedge}A$ is translation and dilation invariant, we have   \begin{align*}
					|h^{\wedge}A|&=hk-h^2+2h+(k+2-z)+1\\& =(h+1)k-h(h-2)-z+3.
				\end{align*}
			\end{enumerate}
			\item[\upshape(xii).]
			\begin{enumerate}
				\item [\upshape (a)] If $x=k-h$, $y \in [k-(h-2),k-2]$, and $z \in [k-(h-4),k]$, then $A=[0,k-h-1] \cup [k-h+1,y-1] \cup [y+1,z-1] \cup [z+1,k+2]$. Note that, $(h+2)k- \dfrac{h(h-1)}{2}-(y+z)+2$ does not belong to $h^{\wedge}A$. Clearly, 
				\begin{multline*}
					h^{\wedge}A= \left[ \dfrac{h(h-1)}{2}, (h+2)k - \dfrac{h(h-1)}{2} -(y+z)+1 \right] \bigcup \left\{ (h+2)k - \dfrac{h(h-1)}{2} -(y+z)+3\right\}. 
				\end{multline*}
				Therefore, $|h^{\wedge}A|= (h+2)k -h(h-1) -(y+z) +3.$
				\item [\upshape (b)] If $x\in [h+1,k-(h+1)]$, $y \in [k-(h-2),k-2]$, and $z \in [k-(h-4),k]$, then $A=[0,x-1] \cup [x+1,y-1] \cup [y+1,z-1] \cup [z+1,k+2]$ and 
				\begin{center}
					$h^{\wedge}A= \left[ \dfrac{h(h-1)}{2}, (h+2)k - \dfrac{h(h-1)}{2} -(y+z)+3 \right],$
				\end{center}
				Therefore, $|h^{\wedge}A|= (h+2)k -h(h-1) -(y+z) +4.$ 
			\end{enumerate}

			\item [\upshape(xiii).] 
			\begin{enumerate}
				\item [\upshape (a)] If $x \in [h+1,k-h-2]$, $y \in [h+3,k-h]$, and $z=k-h+2$, then $A= [0,x-1] \cup [x+1,y-1] \cup [y+1,k-h+1] \cup [k-h+3, k+2].$ Observe that, $hk - \dfrac{h(h-5)}{2}-1$ does not belong to $h^{\wedge}A$. Clearly, 
				\begin{center}
					$h^{\wedge}A= \left[ \dfrac{h(h-1)}{2}, hk- \dfrac{h(h-5)}{2}-2\right] \cup \left\{hk- \dfrac{h(h-5)}{2}\right\}.$
				\end{center} 
				Therefore, $|h^{\wedge}A|= hk-h(h-3).$
				\item[\upshape (b)] If $x,y,z \in [h+1,k-(h-1)]$, then $A=[0,x-1] \cup [x+1,y-1] \cup [y+1,z-1] \cup [z+1,k+2]$ and 
				\begin{center}
					$h^{\wedge}A=\bigg[\dfrac{h(h-1)}{2}, hk - \dfrac{h(h-5)}{2}\bigg].$
				\end{center} Therefore, $|h^{\wedge}A|=hk-h^2+3h+1.$
			\end{enumerate}
		\end{enumerate}
	
	\end{proof}

	For $h\geq 4$, the following proposition \ref{prop 3.16} demonstrates the cardinality of $h^{\wedge}A$ for $A=[0,k+2] \setminus \{1,y,z\}, ~ \text{where} ~ 2 \leq y<z\leq k+1$, with $z-y \geq 2$. Since cardinality of $h^{\wedge}A$ is translation and dilation invariant, Proposition \ref{prop3.17} follows from Proposition \ref{prop 3.16}.
	
	\begin{prop}\label{prop 3.16}
		Let $h \geq 4$ and $k \geq 3h+4$ be positive integers. Let $$A=[0,k+2] \setminus \{1,y,z\}, ~ \text{where} ~ 2 \leq y<z\leq k+1 ~ \text{with} ~ y-z \geq 2.$$
		\begin{enumerate}
			\item[ \upshape (i)] If $y \in [2, h]$ and $z \in [4,h+2]$, then 
			\[|h^{\wedge}A|=
			\begin{cases} 
				hk-h^2-1+(y+z), ~ & ~ \text{if} ~ y \in [2,h-1] ~ \text{and} ~ z \in [4,h+2];\\
				hk-h(h-2), ~ & ~ \text{if} ~ y=h ~ \text{and} ~ z=h+2.
			\end{cases}\]
			\item[ \upshape (ii)] If $y \in [2,h]$ and $ z \in [h+3, k+1] $ , then \[|h^{\wedge}A|=
			\begin{cases}
				hk-h(h-1)+y+1, ~ & ~ \text{if} ~ z \in [h+3, k-(h-1)];\\
				hk-h(h-1)+y, ~ & ~ \text{if} ~ z=k-(h-2);\\
				
				hk-h^2+k+3+(y-z), ~ & ~ \text{if} ~ z \in [k-(h-3),k+1].
			\end{cases}\]
			\item[ \upshape (iii)] If $y=h+1$ and $ z \in [k-(h-2),k+1]$, then \[|h^{\wedge}A|=
			\begin{cases}
				(h+1)k-h(h-1)-z+3, ~ & ~ \text{if} ~ z \in [k-(h-3),k+1];\\
				hk-h(h-2), ~ & ~ \text{if} ~ z=k-(h-2).
			\end{cases}
			\]
			\item[\upshape (iv)] If $y=k-h$ and $z= k-(h-2)$, then $|h^{\wedge}A|= hk-h(h-2)+1.$
			\item[ \upshape (v)] If  $y \in [h+2,k-(h-1)] $ and $ z \in [k-(h-3),k+1]$, then \[|h^{\wedge}A|=
			\begin{cases}
				(h+1)k-h(h-1)-z+4, ~  & ~ \text{if} ~ y \in [h+2,k-h];\\
				(h+1)k-h(h-1)-z+3, ~  & ~ \text{if} ~ y=k-(h-1).
			\end{cases}
			\]
			\item[\upshape(vi)] If $y \in [k-(h-2),k-1]$ and $z \in [k-(h-4),k+1]$, then \[|h^{\wedge}A|= (h+2)k-h^2-(y+z)+5.\]
			\item[ \upshape (vii)] If $y \in [h+1, k-(h+1)]$ and $z \in [h+3,k-(h-1)]$, then  $$|h^{\wedge}A|=
			\begin{cases}
				hk-h^2+2h+1, ~ & ~ \text{if} ~ y=h+1;\\
				hk-h^2+2h+2, ~ & ~ \text{if} ~ y \in [h+2,k-(h+1)].
			\end{cases} .$$

		\end{enumerate}
	\end{prop}
	\begin{proof}
		\begin{enumerate}
			\item [\upshape(i).]
			\begin{enumerate}
				\item[\upshape (a)]	If $y,z \in [2, h+2]$, then $A=\{0\}\cup [2,y-1] \cup [y+1,z-1] \cup [z+1,k+2]$. Evidently, \begin{center}
					$h^{\wedge}A =
					\bigg[\dfrac{(h+1)(h+4)}{2}-(y+z), hk - \dfrac{h(h-5)}{2}\bigg].$
				\end{center} 
				Therefore, we have $|h^{\wedge}A|=	hk-h^2+(y+z)-1.$
				\item[\upshape (b)] If $y=h$ and $z=h+2$, then $A=\{0\} \cup [2,h-1] \cup \{h+1\} \cup [h+3,k+2]$. Observe that, $ \dfrac{h(h+1)}{2}+1$ does not belong to $h^{\wedge}A$. Clearly, \begin{center}
					$h^{\wedge}A= \left\{\dfrac{h(h+1)}{2}\right\} \bigcup \left[\dfrac{h(h+1)}{2}+2, hk-\dfrac{h(h-5)}{2}\right].$
				\end{center}
				Therefore, $|h^{\wedge}A|= hk-h(h-2).$
			\end{enumerate}
			\item [\upshape (ii).]
			\begin{enumerate}
				\item [\upshape (a)] If $y \in [2,h]$ and $z\in [h+3,k-(h-1)]$, then $A=\{0\} \cup [2,y-1] \cup [y+1,z-1] \cup [z+1,k+2]$ and 
				\begin{center}
					$\left[ \dfrac{h(h+3)}{2}-y, hk- \dfrac{h(h-5)}{2}\right]$.
				\end{center}
				Therefore, $|h^{\wedge}A|=hk-h(h-1)+y+1.$
				\item [\upshape (b)] If $y \in [2,h]$ and $z=k-(h-2)$, then $A=\{0\} \cup [2,y-1] \cup [y+1, k-(h-1)] \cup [k-(h-3),k+2].$ Note that, $hk- \dfrac{h(h-5)}{2}-1$ does not belong to $h^{\wedge}A$. Clearly, \begin{center}
					$h^{\wedge}A= \left[\dfrac{(h^2+3h)}{2}-y, hk- \dfrac{h(h-5)}{2}-2 \right] \bigcup \left\{hk- \dfrac{h(h-5)}{2}\right\}.$
				\end{center}
				Therefore, $|h^{\wedge}A|=hk- h(h-1)+y.$
				\item[\upshape (c)]
				If $y \in [2,h]$ and $z \in [k-(h-3),k+1]$,  then \begin{center}
					$A=\{0\} \cup [2,y-1] \cup [y+1,z-1] \cup [z+1,k+2].$ 
				\end{center}  So, \begin{center}
					$h^{\wedge}A =\bigg[ \dfrac{(h^2+3h)}{2}-y, (h+1)k +3-\dfrac{(h-1)(h-2)}{2}-z\bigg].$
				\end{center}
				Therefore we have, 
				$|h^{\wedge}A|= (h+1)k-h^2+(y-z)+3. $

			\end{enumerate}
			\item[\upshape (iii).] 
			\begin{enumerate}
				\item [\upshape (a)] If  $y=h+1$ and $z \in [k-(h-3),k+1]$, then 
				$A=\{0\}\cup [2,h] \cup [h+2,z-1] \cup [z+1,k+2].$
				Note that $\dfrac{h(h+1)}{2} \notin h^{\wedge}A$ and \begin{center}$h^{\wedge}A=  \left\{\dfrac{h(h+1)}{2}-1\right\} \bigcup \bigg[\dfrac{h(h+1)}{2}+1, (h+1)k +3 - \dfrac{(h-1)(h-2)}{2}-z\bigg].$ \end{center} Therefore, we have 
				\begin{center}
					$|h^{\wedge}A|=(h+1)k-h(h-1)-z+3.$
				\end{center}
				\item [\upshape (b)] If  $y=h+1$ and $z=k-(h-2)$, then $A=\{0\}\cup [2,h] \cup [h+2, k-(h-1)] \cup [k-(h-3),k+2].$  Observe that $\dfrac{h(h+1)}{2}$ and $hk- \dfrac{h(h-5)}{2}-1$ do not belong to $h^{\wedge}A$. Clearly,
				\begin{multline*}
					h^{\wedge}A= \left\{\dfrac{h(h+1)}{2}-1\right\} \bigcup \left[\dfrac{h(h+1)}{2}+1,hk- \dfrac{h(h-5)}{2}-2 \right] 
					\bigcup \left\{hk- \dfrac{h(h-5)}{2}\right\}. 
				\end{multline*}
				Therefore, $|h^{\wedge}A|=hk-h(h-2).$
			\end{enumerate}
			\item[\upshape (iv).] If $y = k-h$ and $z=k-(h-2)$, then $A=\{0\} \cup [2,k-(h+1)] \cup \{k-(h-1)\} \cup [k-(h-3),k+2]$. Observe that, $hk- \dfrac{h(h-5)}{2} -1$ does not belong to $h^{\wedge}A$ and 
			\begin{center}
				$\left[\dfrac{h(h+1)}{2}-1, hk- \dfrac{h(h-5)}{2} -2\right] \bigcup \left\{hk- \dfrac{h(h-5)}{2}\right\}$.
			\end{center}
			Therefore, $|h^{\wedge}A|= hk -h(h-2)+1.$
			\item[\upshape (v).] 
			\begin{enumerate}
				\item[\upshape (a)] If $y \in [h+2,k-h]$ and $z \in [k-(h-3),k+1]$, then
				$A=\{0\}\cup [2,y-1] \cup [y+1,z-1] \cup [z+1,k+2].$
				It is evident that, \begin{center}
					$h^{\wedge}A=\bigg[\dfrac{h(h+1)}{2}-1, (h+1)k +3 - \dfrac{(h-1)(h-2)}{2}-z\bigg].$
				\end{center} Therefore,
					$|h^{\wedge}A|=(h+1)k-h(h-1)-z+4.$
				 
				\item [\upshape (b)] If $y=k-(h-1)$ and $z \in [k-(h-3),k+1]$, then $A=\{0\}\cup [2,k-h] \cup [k-(h-2),z-1] \cup [z+1,k+2].$ Note that $(h+1)k - \dfrac{(h-1)(h-2)}{2}-z+2$ does not belong to $h^{\wedge}A$ and \begin{center}
					$h^{\wedge}A=\bigg[\dfrac{h(h+1)}{2}-1, (h+1)k  - \dfrac{(h-1)(h-2)}{2}-z+1\bigg] \bigcup \left\{(h+1)k  - \dfrac{(h-1)(h-2)}{2}-z+3\right\}$.
				\end{center}
				Therefore, $|h^{\wedge}A|=(h+1)k-h(h-1)-z+3.$
			\end{enumerate}
			\item [\upshape(vi).] If $y \in [k-(h-2),k-1]$ and $z \in [k-(h-4),k+1]$, then $A=\{0\} \cup [2,y-1] \cup [y+1,z-1] \cup [z+1,k+2]$ and 
			\begin{center}
				$	h^{\wedge}A= \left[ \dfrac{h(h+1)}{2}-1, (h+2)k - \dfrac{h(h-1)}{2} -(y+z) +3\right]$.
			\end{center}
			 Therefore, $|h^{\wedge}A|= (h+2)k - h^2 -(y+z)+5.$
			\item [\upshape(vii).] 
			\begin{enumerate}
				\item[\upshape (a)] If $y=h+1$, $z \in [h+3,k-(h-1)] $, then $A=\{0\} \cup [2,h] \cup [h+2,z-1] \cup [z+1,k+2].$ It is evident that, $\dfrac{h(h+1)}{2} \notin h^{\wedge}A$ and 
				\begin{center}
					$h^{\wedge}A= \left\{\dfrac{h(h+1)}{2}-1\right\} \bigcup \bigg[\dfrac{h(h+1)}{2}+1, hk - \dfrac{h(h-5)}{2}\bigg].$
				\end{center}
				Therefore, $|h^{\wedge}A|=hk-h^2+2h+1.$
				\item[\upshape(b)]
				If $y \in [h+2,k-(h+1)]$, $z \in [h+4,k-(h-1)] $, then 
				$A=\{0\} \cup [2,y-1] \cup [y+1,z-1] \cup [z+1,k+2].$
				Clearly,
				\begin{center}
					$h^{\wedge}A=\bigg[\dfrac{h(h+1)}{2}-1, hk - \dfrac{h(h-5)}{2}\bigg].$
				\end{center} Therefore, $|h^{\wedge}A|=hk-h^2+2h+2.$
			\end{enumerate}
		\end{enumerate}
	\end{proof}
	\begin{prop}\label{prop3.17}
		Let $h \geq 4$ and $k \geq 3h+2$ be positive integers. Let $$A=[0,k+2] \setminus \{x,y,k+1\}, ~ \text{where} ~ 2 \leq x<y\leq k ~ \text{with} ~ y-x \geq 2.$$
		\begin{enumerate}
			\item[ \upshape (i)] If $ x \in [k-h,k-2] $ and $y \in [k-(h-2),k]$, then \[|h^{\wedge}A|=
			\begin{cases}
				(h+2)k-h^2-(x+y)+3, ~ & ~ \text{if} ~ x \in [k-h,k-2] ~ \text{and} ~ y \in [k-(h-3),k];\\
				hk-h(h-2), ~ & ~ \text{if} ~ x=k-h ~ \text{and} ~ y=k-(h-2).
			\end{cases}
			\]
			\item[ \upshape (ii)] If $x \in [1,k-(h+1)]$ and $ y \in [k-(h-2),k]$ , then \[|h^{\wedge}A|=
			\begin{cases}
				(h+1)k -h(h-1)-y+3, ~ & ~ \text{if} ~ x \in [h+1,k-(h+1)];\\
				(h+1)k-h(h-1)-y+2, ~ & ~ \text{if} ~ x=h;\\
				(h+1)k-h^2+(x-y)+3, ~ & ~ \text{if} ~ x \in [1,h-1].  
			\end{cases}
			\]
			\item[ \upshape (iii)] If $x \in [1,h]$ and $y=k-(h-1)$, then \[|h^{\wedge}A|=	
			\begin{cases}
				hk-h(h-1)+x+1, ~ & ~ \text{if} ~ x \in [1,h-1];\\
				hk-h(h-2), ~ & ~ \text{if} ~ x=h.
			\end{cases}
			\]
			\item[ \upshape (iv)] If  $x=h$ and $y=h+2$, then $|h^{\wedge}A|= hk-h(h-2)+1$.
			\item[ \upshape (v)] If   $ x \in [1,h-1]$ and $y \in [h+1,k-h]$, then \[|h^{\wedge}A|=
			\begin{cases}
				hk-h(h-1)+x+2, ~ & ~ \text{if} ~ y \in [h+2,k-h];\\
				hk-h(h-1)+x+1, ~ & ~ \text{if} ~ y=h+1.
			\end{cases} .\]
			\item[ \upshape (vi)] If $x \in [1,h-2]$ and $y \in [3,h]$, then $|h^{\wedge}A|=hk-h^2+(x+y)+1$.
			\item[ \upshape (vii)] If $x \in [h+1,k-(h+1)]$ and $y \in [h+3,k-(h-1)]$, then  \[|h^{\wedge}A|=
			\begin{cases}
				hk-h^2+2h+1,  ~ & ~ \text{if} ~ y=k-(h-1);\\
				hk-h^2+2h+2, ~ & ~ \text{if} ~ y \in [h+3,k-h].
			\end{cases} \]

		\end{enumerate}
	\end{prop}
	
	\begin{center}
		\textbf{Proofs of the Theorem \ref{one element EIT} - Theorem \ref{3 element EIT gen}}
	\end{center}
	
	\textit{Proof of the Theorem \ref{one element EIT}.} From Theorem \ref{Lemma 1}, it follows that $A \subseteq [0,k]$. So, we assume $A=[0,k] \setminus \{x\}$, where $1 \leq x \leq k-1$. For $x \in \{0,k\}$, $A$ becomes an arithmetic progression. From Proposition \ref{one element EIT prop}, $|h^{\wedge}A|=hk-h^2+2$ holds when $x \in \{1,k-1\}$. \\
	
	Proofs of the Theorem \ref{two element EIT h=3} - Theorem \ref{3 element EIT gen} follows similarly.


\bibliographystyle{amsplain}

\begin{thebibliography}{5}
	
	
	\bibitem{Plagne06-I} A. Plagne, \textit{Optimally small sumsets in groups, I. The supersmall sumset property, the $\mu_G^{(k)}$ and the $\nu_G^{(k)}$ functions}, Unif. Distrib. Theory \textbf{1} (2006), 27–44.
	
	\bibitem{Plagne06-II} A. Plagne, \textit{Optimally small sumsets in groups, II}. The hypersmall sumset property and
	restricted addition, Unif. Distrib. Theory \textbf{1} (1) (2006), 111-124.

	

	\bibitem{Grynkiewicz2009} D. J. Grynkiewicz, \textit{A step beyond Kemperman’s structure theorem}, Mathematika \textbf{55} (2009), 67–114.
	

	
	
	
	\bibitem{Cauchy} A. L. Cauchy, \textit{Recherches sur les nombres}, J. \'Ecole Polytech \textbf{9} (1813), 99--116.
	
	\bibitem{Davenport} H. Davenport, \textit{On the addition of residue classes}, J. Lond. Math. Soc. \textbf{10} (1935),
	30--32.

	\bibitem{Mann} H. B. Mann, \textit{Addition Theorems: the addition theorems of group theory and number
		theory}, Wiley-Interscience, New York, 1965.
	
	
	\bibitem{Daza2023}  D. Daza, M. Huicochea, C. Martos and C. Trujillo, A Freiman-type theorem for restricted sumsets \textit{Int. J. Number Theory} \textbf{19}(10) (2023), 2309--2332.
	\bibitem{DSH1994}J. A. Dias da Silva and Y. O. Hamidoune, \textit{Cyclic space for Grassmann deriva-
		tives and additive theory}, Bull. London Math. Soc. \textbf{26} (1994), 140--146.
	\bibitem{Kemperman60} J. H. B. Kemperman, \textit{On small sumsets in an abelian group}, Acta Arith. \textbf{103}(1960), 63–88.
	
	\bibitem{GAF1973} G. A. Freiman, {\it Foundations of structural theory of set addition}, vol. {\bf 37}, Translations of Mathematical Monographs, American Mathematical Society, Provedience, R.I., 1973.
	
	\bibitem{FreimanLowPit}G. Freiman, L. Low, and J. Pitman, \textit{Sumsets with distinct summands and the conjecture of
		Erd\"os–Heilbronn on sums of residues}, Ast\`erisque \textbf{258} (1999), 163--172.
	
%
%
	
	\bibitem{Lev1999} V. F. Lev, \textit{The structure of multisets with a small number of subset sums},
	Astérisque \textbf{258}(1999), 179–186.
	
	\bibitem{VFLev2000} V. F. Lev,
	\textit{Restricted Set Addition in Groups I: The Classical Setting}, J. of the London Math. Soc. \textbf{62} (2000).
	
%
	
	\bibitem{ElKerPlagne2003} S. Eliahou, M. Kervaire, and A. Plagne, \textit{Optimally small sumsets in finite Abelian groups},
	J. Number Theory \textbf{101} (2003), 338–348.
	

	
	\bibitem{MohanPandey2023-I} Mohan and R. K. Pandey, \textit{Extended inverse theorems for $h$-fold sumsets in integers, Contributions to Discrete Mathematics, \textbf{18} No. 2 (2023) 129-145}
	
	\bibitem{Mohan_Pandey_2023_restricted_sumset} Mohan and R. K. Pandey, Extended inverse theorems for restricted sumset  in integers, Bulletin of the Korean Mathematical Society, \textbf{61} (2024), No. 5, pp. 1339–1367.
	

	\bibitem{Nathanson1996}M.B. Nathanson, Additive Number Theory: Inverse Problems and the Geometry of Sumsets, Springer, 1996.
	
	\bibitem{ANR95}N. Alon, M. B. Nathanson, and I. Z. Ruzsa, \textit{Adding distinct congruence classes modulo
		a prime}, Amer. Math. Monthly \textbf{102} (1995), 250–255.
	
	\bibitem{ANR96} --------- , \textit{The polynomial method and restricted sums of congruence classes}, J. Number
	Theory \textbf{56} (1996), 404–417			
	\bibitem{Erdos-heibronn64} P. Erdős and H. Heilbronn, \textit{On the addition of residue classes (mod p)}, Acta Arith. \textbf{9}
	(1964), 149–159.
	
	\bibitem {TaoVu2006} T. Tao, V. H. Vu, Additive Combinatorics, Cambridge studies in advanced mathematics, Vol. 105, {Cambridge University Press, Cambridge, 2006\/}.
	
	
	
	
	\bibitem{TangXing2021} M. Tang and Y. Xing, \textit{Some inverse results of sumsets}, Bull. Korean Math. Soc., \textbf{58} (2021), No. 2, pp. 305-313.
	
	
	

	
	
\end{thebibliography}

\end{document}